\documentclass[10pt,reqno]{amsart}

\usepackage{amscd}
\usepackage{amsthm}
\usepackage{amssymb}

\usepackage{enumerate,array,mathrsfs,a4wide}
\usepackage{graphicx}
\usepackage{xcolor}

\usepackage{rotating}
\usepackage{bbm}
\usepackage[all,cmtip]{xy}

\usepackage[latin1]{inputenc}
\usepackage{dsfont}

\usepackage{hyperref}

\newtheorem{theorem}{Theorem}[section]

\newtheorem{lemma}[theorem]{Lemma}
\newtheorem{proposition}[theorem]{Proposition}
\newtheorem{corollary}[theorem]{Corollary}

\newtheorem*{theorem*}{Theorem}

\theoremstyle{definition}
\newtheorem{definition}[theorem]{Definition}

\newtheorem{example}[theorem]{Example}
\newtheorem{question}[theorem]{Question}

\theoremstyle{remark}
\newtheorem{remark}[theorem]{Remark}

\numberwithin{equation}{section}

\newcommand{\ow}{\omega}

\newcommand{\C}{{\mathbb{C}}}
\newcommand{\R}{{\mathbb{R}}}
\newcommand{\Q}{{\mathbb{Q}}}
\renewcommand{\P}{{\mathbb{P}}}
\renewcommand{\H}{{\mathbb{H}}}
\newcommand{\Z}{{\mathbb{Z}}}

\renewcommand{\epsilon}{\varepsilon}
\renewcommand{\phi}{\varphi}
\renewcommand{\theta}{\vartheta}

\newcommand{\w}{\wedge}

\DeclareMathOperator{\Aut}{Aut}

\DeclareMathOperator{\grad}{grad}

\DeclareMathOperator{\Hom}{Hom}
\DeclareMathOperator{\id}{Id}

\DeclareMathOperator{\ind}{ind}
\DeclareMathOperator{\lk}{lk}

\DeclareMathOperator{\OB}{OB}

\DeclareMathOperator{\sgn}{sgn}

\DeclareMathOperator{\Span}{span}
\DeclareMathOperator{\Spec}{Spec}

\begin{document}
\title[Left-handed twists]{Non-fillable invariant contact structures on principal circle bundles and left-handed twists}

\author[River Chiang]{River Chiang}
\address{
River Chiang,
Department of Mathematics, National Cheng Kung University,
Tainan 701, Taiwan
}
\email{riverch@mail.ncku.edu.tw}

\author[Fan Ding]{Fan Ding}
\address{
Fan Ding,
School of Mathematical Sciences and LMAM, Peking University, Beijing 100871, P.~R.~China
}
\email{dingfan@math.pku.edu.cn}

\author[Otto van Koert]{Otto van Koert}
\address{
Otto van Koert,
Department of Mathematics and Research Institute of Mathematics, Seoul National University\\
Building 27, room 402\\
San 56-1, Sillim-dong, Gwanak-gu, Seoul, South Korea\\
Postal code 151-747
}
\email{okoert@snu.ac.kr}


\subjclass[2010]{53D35, 32Q65, 57R17}

\keywords{Invariant contact structures, fractional twists, Dehn twists, overtwistedness, open books}

\begin{abstract}
We define symplectic fractional twists, which generalize Dehn twists, and use these in open books to investigate contact structures.
The resulting contact structures are invariant under a circle action, and share several similarities with the invariant contact structures that were studied by Lutz and Giroux.
We show that left-handed fractional twists often give rise to non-fillable contact manifolds.
These manifolds are in fact ``algebraically overtwisted'', yet they do not seem to contain bLobs, nor are they directly related to negative stabilizations.
We also show that the Weinstein conjecture holds for the non-fillable contact manifolds we construct, and we investigate the symplectic isotopy problem for fractional twists. 
\end{abstract}

\maketitle

\section{Introduction}
About ten years ago, Giroux suggested that the notion of overtwistedness could be generalized to higher dimensions using open books: a negative (or left-handed) stabilization should be overtwisted.
Since then, various versions of overtwistedness in higher dimensions have been defined, most notably the plastikstufe, a special kind of family of overtwisted disks defined in \cite{Niederkrueger:Plastikstufe}, and the bLob, the bordered Legendrian open book defined in \cite{Massot:weak_strong_fillability}.
These objects reduce to overtwisted disks in dimension $3$.
Furthermore, existence of plastikstufes or bLobs obstructs fillability, implies algebraic overtwistedness and the existence of a contractible periodic Reeb orbit, see \cite{Albers:Weinstein_PS,Massot:weak_strong_fillability}.
These are all features in common with overtwisted contact $3$-manifolds.
In addition, plastikstufes exhibit some flexibility properties as shown in \cite{Murphy:loose_plastik}, so they seem to be the right generalization.

On the other hand, it was also shown that negative stabilizations have some aspects of overtwisted manifolds; for instance, negative stabilizations do not admit fillings, are algebraically overtwisted, and also have contractible periodic Reeb orbits, \cite{BvK,Massot:weak_strong_fillability}.
It is therefore reasonable to look for plastikstufes and bLobs in negative stabilizations, but it seems to be difficult to find these objects in these contact manifolds.

In this paper, we take a different point of view on this question.
We shall show that negative stabilizations are part of a much larger class of ``negatively twisted'' contact manifolds.
In order to define these, consider a Liouville domain $W$ whose boundary is a prequantization bundle $P$ over an integral symplectic manifold $(Q,\omega)$.
A collar neighborhood of the boundary looks like $(P\times I,d(e^t \lambda) \,)$, so define the symplectomorphism $\tau:(p,t)\mapsto (Fl^{R_\lambda}_{f(t)}(p),t)$, where $f: I\to \R$ is a smooth function such that $f(0)=2\pi/\ell$ for some positive integer $\ell$ and $f(1)=0$.

In some cases $\tau$ extends to the whole of $W$; 
for instance, if $\ell=1$. This gives a so-called fibered Dehn twist, which was already considered by Biran and Giroux, \cite{Biran_Giroux:fibered_Dehn}. 
We will explain some sufficient conditions and an explicit procedure to define such an extension in Section~\ref{sec:fractional_twist} by using a covering trick. Assume for now that we can construct such an extension in some definite way and call it a right-handed fractional twist of power $\ell$: this notion generalizes Dehn twists and also fibered Dehn twists.
Consider the contact open book $\OB(W,\tau^{\pm 1})$ with page $W$ and monodromy either a right-handed fractional twist $\tau$ or a left-handed one, $\tau^{-1}$.

We shall show that the contact manifolds constructed this way are principal circle bundles over smooth manifolds, and that the contact structure is invariant under the circle action.
\begin{theorem}
\label{thm:result invariant}
Let $(W,\Omega=d\lambda)$ be a Liouville domain such that $P:=\partial W$ is a prequantization bundle over a symplectic manifold $(Q,k\omega)$, where $\omega$ is a primitive, integral symplectic form, and $k\in \Z_{>0}$.
Suppose $pr:\tilde W\to W$ is an adapted $\ell$-fold cover (see Definition~\ref{def:adapted_cover}; such a cover trivially exists for $\ell=1$).
Then one can define a right-handed fractional twist $\tau$ of power $\ell$ on the cover $\tilde W$.
In particular, a fibered Dehn twist exists on $W$.

Furthermore, we have the following results for contact open books with the above monodromies.
\begin{enumerate}
\item The contact open book $\OB(\tilde W,\tau)$ is a prequantization bundle over the symplectic manifold
$$
M_+=\left(
P\times_{S^1,+}D^2
\right)
\cup_\partial W,
$$
where $P\times_{S^1,+}D^2$ denotes the associated disk bundle which is a concave filling for $P$.
In addition, this contact manifold is convex fillable.
\item The contact open book $\OB(\tilde W,\tau^{-1})$ is a principal circle bundle over the smooth manifold
$$
M_-=\left(
P\times_{S^1,-}D^2
\right)
\cup_\partial W,
$$
where $P\times_{S^1,-}D^2$ denotes the associated disk bundle dual to $P\times_{S^1,+}D^2$.
Furthermore, the contact structure on this contact open book is $S^1$-invariant, and the almost dividing set is contactomorphic to the prequantization bundle $P$.
\end{enumerate}
\end{theorem}
Before we continue, let us point out that a particularly nice class of Liouville manifolds $W$ with the above properties can be obtained from the following construction.
Let $(M,\omega_M)$ be an integral symplectic manifold, and suppose that $Q$ is a Donaldson type symplectic hypersurface that is Poincar\'e dual to $k[\omega_M]$.
For $k$ sufficiently large, $W:=M-\nu_{M}(Q)$ carries the structure of a Weinstein domain, and its completion is a Weinstein manifold whose end is the positive part of the symplectization of a prequantization bundle over $Q$.

We shall show that contact open books with left-handed twists as monodromy can be non-fillable, and below we give some criteria for this.
If we replace the word non-fillable by ``overtwisted'', the result becomes somewhat similar to the Giroux criterion for overtwistedness.
In fact, it is even more closely related to another result of Giroux, namely a result on invariant contact structures, \cite{Giroux:circle_bdls}. 
Our result is much weaker, but it does hold in higher dimensions.
\begin{theorem}[Rough version of Giroux criterion in higher dimensions]
\label{thm:result}
As in Theorem~\ref{thm:result invariant}, let $(W^{2n-2},\Omega=d\lambda)$ be a Liouville domain with prequantization boundary $P$. 
Assume that $\tilde W \to W$ is an adapted $\ell$-fold cover, and let $\tau$ denote a right-handed fractional twist of power $\ell$, defined on $\tilde W$.

Consider the invariant contact structure on $Y=\OB(\tilde W,\tau^{-1})$, and assume that one of the following conditions holds.
\begin{itemize}
\item $\tau$ is a fractional twist of power $\ell>1$. 
\item $\tau$ is a fibered twist ($\ell=1$), $k>1$, and the inclusion $i:P=\partial W\to W$ induces an injection on $\pi_1$.
\item $\tau$ is a fibered twist ($\ell=1$) on $\tilde W=W=T^*\H \P^m, ~T^*Ca\P^2$, the cotangent bundles of quaternionic projective space and the Cayley plane.
\item $\tau$ is a fibered twist ($\ell=1$), $n\geq 3$, $\pi_1(Q)=0$, $k=1$, $c_1(W)=0$ and $c_1(Q)=c[\omega]$ with $c\leq n-\frac{\max \ind+3 }{2}$.
Here $\max \ind$ is the maximal index of a Morse function on $W$ that is convex near the boundary.
\end{itemize}
Then $Y$ is not convex semi-positively fillable.
\end{theorem}
To clarify the condition on the maximal index of a Morse function, we point out that if $W^{2n-2}$ is a Weinstein domain, then one can find an $\Omega$-convex Morse function, which satisfies $\max \ind\leq n-1$.

We remark that some condition is necessary for non-fillability in the case of a fibered twist. Indeed, take $W=D^{2n-2}$. 
Then $P=\partial W$ is a prequantization bundle over $\C \P^{n-2}$, so $c=n-1>0$, and $k=1$.
In this case a fibered Dehn twist is symplectically isotopic to the identity relative to the boundary, so $Y=\OB(D^{2n-2},\tau^{-1})$ is contactomorphic to $(S^{2n-1},\xi_0)$, which is of course fillable.
On the other hand, this condition on the Chern class can certainly be relaxed, as the case $W=T^*\H \P^m, ~T^*Ca\P^2$ shows. 

We also want to point the reader's attention to the case of $\dim W=2$, where $P$ is a collection of circles.
In this case, various questions about tightness and overtwistedness have been addressed by Giroux in \cite{Giroux:circle_bdls}.
We just mention one particular case, namely $\OB(T^*S^1,\tau_{Dehn}^2)\cong \OB(T^*S^1,\tau_{fibered})$, which is overtwisted, and in particular not fillable.

The idea of the proof of Theorem~\ref{thm:result} is to construct a $1$-dimensional family of holomorphic planes in an assumed filling with one ``boundary'' component on the boundary of the filling. Such a family arises rather naturally in the case of a left-handed twist.
The conditions we impose prevent breaking of this family, and also exclude this family from having its other boundary component on the boundary of the filling.
Since sphere bubbles are prevented by the semi-positivity assumption, there cannot be another boundary component at all, and we arrive at a contradiction.
The holomorphic curve methods that we use also apply to the $3$-dimensional case (so $n=2$), but the resulting statements are weaker than those of Giroux.

The methods used in the proof of Theorem~\ref{thm:result} also imply the following result.
\begin{corollary}
\label{cor:weinstein_conj}
If the contact manifold $Y$ obtained in Theorem~\ref{thm:result} satisfies the conditions given there, then the Weinstein conjecture holds for that particular contact structure.
\end{corollary}

Finally, these methods can also be used to tackle the symplectic isotopy problem for fibered twists, first considered by Biran and Giroux, \cite{Biran_Giroux:fibered_Dehn}, and also for fractional twists.
\begin{theorem}
\label{thm:symplectic_isotopy}
Let $W$ be a Weinstein domain admitting a right-handed fractional twist $\tau$ of power $\ell$, and suppose that $\OB(W,\tau^{-1})$ has no convex semi-positive filling (shown for instance by Theorem~\ref{thm:result}).
Assume that $Y_+=\OB(W,\tau)$ admits a convex, semi-positive symplectic filling (a sufficient condition is given in Lemma~\ref{lemma:fillability_BW}).
Then for all $N \in \Z_{>0}$, the contact manifold $\OB(W,\tau^{-N})$ is not convex semi-positively fillable.

In particular, $\tau^N$ is not symplectically isotopic to the identity relative to the boundary.
\end{theorem}
To make the statement about powers of fractional twists, we use Avdek's cobordism techniques, \cite{Avdek:Liouville}.
The symplectic isotopy problem for fibered twists was also addressed in \cite{CDvK:right-handed}, and the methods in that paper are somewhat simpler than those employed here.
Also, in several of the cases the obtained twists are not even smoothly isotopic to the identity relative to the boundary.

\begin{remark}
If we assume that contact homology algebra can be defined, using for instance the not yet finished polyfold techniques, then one can show that all contact manifolds from Theorem~\ref{thm:result} are actually algebraically overtwisted as in \cite{Bourgeois_Niederkrueger:AOT} and \cite{BvK}, meaning that contact homology algebra vanishes.
We will give arguments for these conjectural statements in Section~\ref{sec:HC=0}.
This should also imply that there are no weak fillings at all, cf.~\cite[Theorem 5]{Latschev_Wendl}, but again this assumes polyfolds or some kind of non-classical transversality argument.

Concerning weak fillings, one has the general fact that weak fillings can be deformed into strong fillings if $H^2(Y;\R)=0$, see \cite[Remark 2.11]{Massot:weak_strong_fillability}.
\end{remark}

\subsection*{Plan of the paper}
The paper is organized as follows. 
In Section~\ref{sec:def} we give definitions and discuss some background on the type of filling obstructions that we shall use.
In Section~\ref{sec:invariant_contact} we discuss and construct the invariant contact structures from Theorem~\ref{thm:result invariant}.
The remainder of the paper is used for the proof of Theorem~\ref{thm:result} and its corollaries.
Section~\ref{sec:indices} contains index computations.
Section~\ref{sec:holomorphic_plane} is the main technical part of the paper: we construct a rigid holomorphic plane, and discuss transversality and uniqueness.
In Section~\ref{sec:other_curves} we look at other holomorphic curves, and show that the assumptions of Theorem~\ref{thm:result} imply that no other rigid curves exist.
Finally, we combine all ingredients in Section~\ref{sec:wrapup} to prove Theorem~\ref{thm:result}, Corollary~\ref{cor:weinstein_conj} and Theorem~\ref{thm:symplectic_isotopy}.

\subsection*{Acknowledgements}
This project grew out of one of the questions asked at the AIM workshop on ``Contact topology in higher dimensions''. OvK would like to thank AIM for their hospitality.
We thank Klaus Niederkr\"uger, Chris Wendl and Urs Frauenfelder for useful comments and discussions.
RC is partially supported by the NSC grant 101-2115-M-006-003 and NCTS(South), Taiwan;
FD is supported by grant no. 10631060 of the National Natural Science Foundation of China;
OvK is supported by the NRF Grant 2012-011755 funded by the Korean government.

\section{Definitions and setup}
\label{sec:def}
Let $(Q,\omega)$ be a symplectic manifold with integral symplectic form, i.e.~$[\omega]\in H^2(Q;\Z)$.
According to \cite{Kobayashi:circle_bundle}, there is a complex line bundle $L$ over $Q$ with $c_1(L)=[\omega]$. If $H^2(Q;\Z)$ is torsion free, then this line bundle is unique up to isomorphism.
The associated principal $S^1$-bundle $\Pi: P \to Q$ carries a contact form $\theta$, the so-called {\bf Boothby--Wang form}, which is a connection $1$-form on $P$ whose curvature form equals $-2\pi\ow$,
\begin{equation}
\label{eq:d_connection}
d\theta = -2\pi\Pi^*\ow.
\end{equation}
The vector field $R_\theta$ generating the principal $S^1$-action satisfies the following equations
$$
\iota_{R_\theta} \theta=1,\quad\iota_{R_\theta}d\theta=0,
$$
since $\theta$ is a connection form. On the other hand, these are also the equations defining the Reeb vector field for $\theta$.
The resulting principal circle bundle is called the
\textbf{Boothby--Wang bundle} or {\bf prequantization bundle} associated with $(Q,\omega)$.
It is useful to think of $Q$ as the quotient space of the prequantization bundle $P$ by the $S^1$-action.

\subsection{Fillings}
Let $(C,\Omega)$ be a compact symplectic manifold with boundary.
We call a boundary component $Y$ of $C$ {\bf convex} if there is a Liouville vector field $X$ (so $\mathcal L_X \Omega=\Omega$) on a collar neighborhood of $Y$ that points outward.
We call the boundary component $Y$ {\bf concave} if there is a Liouville vector field pointing inward.
Note that the Liouville vector field induces a contact form on the boundary, given by $\lambda_Y=(i_X \Omega)|_{Y}$.

If $Y$ is a convex or concave boundary component, then a collar neighborhood $\nu_C(Y)$ is symplectomorphic to a piece of a symplectization, namely $( [-\epsilon,\epsilon]\times Y,d(e^t \lambda_Y)\,)$, and the Liouville field $X$ corresponds to the vector field $\partial_t$.
We will often attach a {\bf convex end}, given by $( [\epsilon,\infty[ \times Y,d(e^t \lambda_Y)\,)$, to a convex boundary in order to obtain a symplectic manifold with a Liouville vector field that is forward complete.
Similarly, we can define a {\bf concave end}.

\begin{definition}
A {\bf compact symplectic cobordism} is a compact symplectic manifold $(C,\Omega)$ whose boundary components are all convex or concave.
A {\bf complete symplectic cobordism} is then obtained by attaching convex and concave ends to the boundary components.
\end{definition}

We come now to symplectic fillings, which can be thought of as compact symplectic cobordisms with only a convex, or only a concave boundary component.
\begin{definition}
A {\bf convex symplectic filling} for a contact manifold $(Y,\xi=\ker \lambda_Y)$ is a connected, compact symplectic manifold $(C,\Omega)$ with boundary $Y$ such that
\begin{itemize}
\item $(Y,\lambda_Y)$ is a convex boundary of $(C,\Omega)$ with $\lambda_Y=(i_X \Omega)|_{Y}$.
\item $(Y,\lambda_Y)$ is oriented as the boundary of $C$ using that $X$ points outward.
\end{itemize} 
Convex symplectic fillings are also known as {\bf strong fillings}.
A {\bf Liouville filling} or {\bf Liouville domain} is a convex symplectic filling with a globally defined Liouville vector field $X$.
\end{definition}
In particular, Liouville domains are exact symplectic manifolds.
One of the nicest symplectic fillings and a basic building block for many constructions is formed by so-called Weinstein manifolds, a symplectic analogue of Stein manifolds.
\begin{definition}
A {\bf compact Weinstein manifold} or {\bf Weinstein domain} consists of a Liouville domain $(W,\Omega,X)$ together with a Morse function $f:W\to \R$ such that $\partial W$ is a regular level set of $f$, and $X$ is gradient-like for $f$, i.e.~$X(f)>0$ except at critical points, where it has a standard form (like a gradient vector field).
\end{definition}
An immediate corollary of the definition is that $\partial W$ is of contact type.
We say that $\partial W$ is {\bf Weinstein fillable}.
By results of Eliashberg, one can deform {\bf Weinstein manifolds}, which are obtained by attaching a symplectization to the boundary, into Stein manifolds~\cite{Cieliebak:Stein_Weinstein}.
Clearly, Weinstein manifolds are special cases of Liouville manifolds which are in turn special cases of convex symplectic manifolds.
\begin{remark}
When talking about symplectic cobordisms and fillings, we will often drop the adjectives \emph{compact} and \emph{complete}, since either case can be converted into the other, by either attaching a symplectization piece, or by restricting to a compact subset.
\end{remark}

\subsubsection{Weak symplectic fillings}
The notion of weak symplectic filling in higher dimensions has recently been defined in a satisfactory way by Massot, Niederkr\"uger and Wendl in \cite{Massot:weak_strong_fillability}, see also \cite{DingGeiges:circle_bdls}.
We take their definition.
\begin{definition}[Massot, Niederkr\"uger, Wendl]
Let $(Y^{2n-1},\xi)$ be a cooriented contact manifold.
A {\bf weak symplectic filling} for $(Y,\xi)$ is a compact symplectic manifold $(F,\Omega)$ with boundary such that
\begin{itemize}
\item $\partial F=Y$ as oriented manifolds,
\item for a positive contact form $\alpha$ with $\xi=\ker \alpha$ we have
$$
\alpha\w (t d\alpha+\Omega|_{\xi})^{n-1}>0
$$
for all $t\geq 0$.
\end{itemize}
\end{definition}
We will also need the following definition,
\begin{definition}[Massot, Niederkr\"uger, Wendl]
Let $(F^{2n},J)$ be an almost complex manifold with boundary. We say that a contact manifold $(Y^{2n-1},\xi)$ is the {\bf tamed pseudoconvex boundary} of $(F,J)$ if $Y=\partial F$ where we orient $Y$ as the boundary of $F$, and
\begin{itemize}
\item the contact structure $\xi$ is the field of $J$-complex tangencies to $Y$, that is $\xi=TY \cap J TY$.
\item there is a symplectic form $\Omega$ on $F$ taming $J$, and
\item $Y$ is $J$-convex, meaning that for all positive contact forms $\alpha$ defining $\xi$ as an oriented hyperplane, we have
$$
d\alpha(v,Jv )>0 \text{ for all }v\in \xi-0.
$$
\end{itemize}
\end{definition}

\begin{theorem}[Massot-Niederkr\"uger-Wendl]
A compact symplectic manifold $(F,\Omega)$ is a weak filling of $(Y,\xi)$ if and only if there is an almost complex structure $J$ such that
\begin{itemize}
\item $J$ is tamed by $\Omega$,
\item $(Y,\xi)$ is the tamed pseudoconvex boundary of $(F,J)$.
\end{itemize}
\end{theorem}

These definitions allow one to use holomorphic curve machinery, and with these tools one can show that under certain conditions, such as negative stabilization or the existence of a bLob, weak fillings do not exist.
See \cite{BvK,Massot:weak_strong_fillability} for the definitions of negative stabilizations and bLobs.
As is usual, we will compactify the moduli space of holomorphic curves, \cite{BEHWZ:compactness}. Such a compactification can include nodal curves with multiply covered components. These cause trouble for regularity arguments, and hence it is not clear what structure the moduli space actually has.
To avoid these issues, one should be able to appeal to the polyfold machinery. However, this theory has not yet been completed, so we impose additional assumptions to deal with regularity.
\begin{definition}
A symplectic manifold $(F^{2n},\Omega)$ is called {\bf semi-positive} if every class $A\in \pi_2(F)$ with $\langle [\Omega],A \rangle >0$ and $\langle c_1(F),A \rangle \geq 3-n$ has non-negative Chern number.
\end{definition}
Note that this is trivially satisfied if $n\leq 3$.

\subsection{Contact open books}
\label{sec:contact_OB}
We follow Giroux' original construction of contact open books, which slightly modifies an idea due to Thurston and Winkelnkemper, see \cite{Giroux:ICM2002}.
Let $(W,d\lambda)$ be a Liouville domain with contact type boundary $P:=\partial W$, and suppose that $\psi:W \to W$ is a symplectomorphism that is the identity in a neighborhood of $\partial W$. 
By a lemma of Giroux, we can assume that $\psi^*\lambda=\lambda-dU$ for a positive function $U$.
The `stretched' mapping torus
$$
Map(W,\psi):=W \times \R / \sim
$$
where $(x,\phi)\sim (\psi(x),\phi+U(x)\,)$, carries the contact form $d\phi+\lambda$.
The set $W$ will be called the {\bf page}.
Note that $U$ is constant on each boundary component; we shall assume that the value of this constant equals $U_c$ on each boundary component.

\subsubsection{Binding}
Let $(P,\lambda_P)$ denote the contact type boundary of the page $W$.
We construct a closed contact manifold by gluing the set $P\times D^2_{r_0}$ to the mapping torus $Map(W,\psi)$.
On the set $P\times D^2_{r_0}$, we use a standard model for the contact structure. Choose functions $h_1$ and $h_2$ as indicated in Figure~\ref{fig:functions_binding}.
These functions satisfy $h_1(r) h_2'(r)-h_2(r) h_1'(r)>0$ and $h_1(r)>0$ if $r>0$, so the following form is a contact form,
\begin{equation}
\label{eq:form_near_binding}
\alpha=h_1(r) \lambda_P+h_2(r) d\phi.
\end{equation}

\begin{figure}[htp]
\def\svgwidth{0.65\textwidth}%
\begingroup\endlinechar=-1
\resizebox{0.65\textwidth}{!}{%
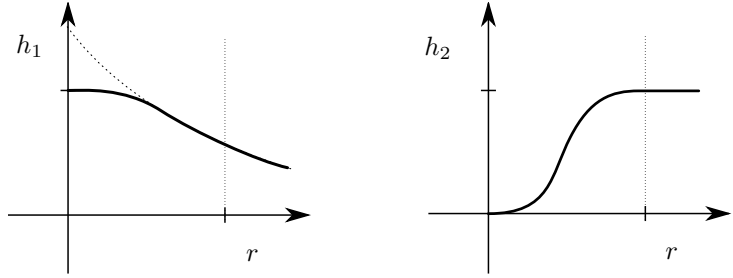%
}\endgroup
\caption{Functions for the contact form near the binding}
\label{fig:functions_binding}
\end{figure}

Note that a neighborhood of the boundary of $Map(W,\psi)$ is diffeomorphic to $P\times [-\delta,0] \times S^1$ for some $\delta>0$, because $\psi$ is the identity in a neighborhood of the boundary of $W$.
Let $A_{r_0,r_0+\delta}$ denote an annulus with inner radius $r_0$ and outer radius $r_0+\delta$, and define the gluing map
\[
\begin{split}
\psi_G: P\times A_{r_0,r_0+\delta}\subset P\times D^2_{r_0} &\longrightarrow \nu_{Map(W,\psi)}(\partial Map(W,\psi)\,) \subset Map(W,\psi)\\
(p;r,\phi) & \longmapsto (p,r_0-r,\frac{U_c \phi}{2\pi}).
\end{split}
\]

A {\bf contact open book} is a contact manifold obtained by gluing the above sets together.
We define
$$
\OB(W,\psi^{-1})=P\times D^2_{r_0} \cup_{\psi_G} Map(W,\psi),
$$
and we call $\psi^{-1}$ the {\bf monodromy} of the open book.
The subset $P\times \{ 0 \}$ is called the {\bf binding}.

\begin{remark}
The inverse in the definition of a contact open book is needed due to our conventions for a mapping torus.
\end{remark}

\subsection{Fractional fibered Dehn twists}
\label{sec:fractional_twist}
Let $(P,\lambda_P)$ be a prequantization bundle over a symplectic manifold $(Q,k\omega)$, where $[\omega]$ represents a primitive class in $H^2(Q;\Z)$, and $k$ is a positive integer.
Suppose furthermore that $(W^{2n},\Omega=d\lambda,X)$ is a Liouville filling for $(\partial W=P,\lambda_P)$.
Denote the inclusion $P\subset W$ by $j$. 
A collar neighborhood of the boundary of $W$ looks like a piece of a symplectization
$$
\nu_{W}(P)=(P\times [a,b], d(e^t \lambda_P ) \,).
$$
Here $P\times \{ b\}$ denotes the boundary.
We will usually take $[a,b]=[0,1]=I$, although this is not important.
Define $W_{in}=W-\nu_{W}(P)$; we obtain the decomposition
$$
W=W_{in}\cup_\partial P\times [a,b].
$$
We shall call $P\times [a,b]$ the {\bf margin} of the page: it carries the Liouville form $\lambda=e^t\lambda_P$.
The set $W_{in}$ will be called the {\bf content} of the page.

Fix a smooth function $f_m:[a,b]\to \R$ (the $m$ stands for monodromy) such that $f_m(a)=2\pi$ and $f_m(b)=0$.
We shall refer to this function as {\bf twisting profile}.
Define
\begin{eqnarray*}
\tau: P\times [a,b] & \longrightarrow & P\times [a,b] \\
(p,t) & \longmapsto & (Fl^{R_{P}}_{f_m(t)}(p),t),
\end{eqnarray*}
where $R_{P}$ is the Reeb field of $\lambda_P$.
Extend the map $\tau$ to be the identity on the content $W_{in}$.
\begin{lemma}
The map $\tau$ is a symplectomorphism satisfying
$$
\tau^* \lambda=\lambda-dU,
$$
where
\begin{equation}
\label{eq:formula_u}
U(p,t)=U(t)=C-\int_{s=a}^t f_m'(s)e^sds.
\end{equation}
\end{lemma}
\begin{proof}
On a collar neighborhood of the boundary we have $\lambda=e^t\lambda_P$.
By the Cartan formula we compute
\[
\begin{split}
\mathcal L_{f_mR_{P}} \lambda & =d(e^tf_m(t)\,)+i_{f_mR_{P} }\left( e^t dt\w \lambda_P +e^t d\lambda_P \right) \\
&=f_m'(t)e^tdt.
\end{split}
\]
Since $\int_{s=0}^1 \frac{d}{ds}{Fl^{f_mR_{P}}_s}^* \lambda \, ds=\tau^*\lambda-\lambda$, we find the above Formula~\eqref{eq:formula_u} for $U$.
\end{proof}
The following notion was introduced by Biran and Giroux, \cite{Biran_Giroux:fibered_Dehn}.
\begin{definition}
We call the symplectomorphism $\tau$ a {\bf right-handed fibered Dehn twist}.
\end{definition}

Next, take a positive integer $\ell$ dividing $k$ and assume that $P$ is covered by $\tilde P$, a prequantization bundle over $(Q,\frac{k}{\ell}\omega)$ by an $\ell$-fold covering of the form
\begin{equation}
\label{eq:cover_BW_bdl}
\begin{split}
pr_\ell: \tilde P & \longrightarrow P \\
p & \longmapsto p\otimes \ldots \otimes p.
\end{split}
\end{equation}
We remark that this can always be done if $H^2(Q;\Z)$ is torsion free.
Note that this covering shows that $\pi_1(P)$ has a subgroup of index $\ell$.
We want to extend this covering to the filling, and for that purpose we introduce the following definition,
\begin{definition}
\label{def:adapted_cover}
We call a covering $pr_W:\tilde W \to W$ {\bf adapted} to the covering $pr_\ell: \tilde P \to P$ if $pr_W$ restricts on a collar neighborhood of the boundary $\nu_{\tilde W}( \partial \tilde W)=\tilde W \times I$ to $pr_\ell \times \id$,
$$
pr_W|_{\partial \tilde W\times \{t_0 \}}=pr_\ell \times \{ t_0\}.
$$
\end{definition}

If the Liouville filling $W$ is Weinstein of dimension at least $6$, then we can always find an adapted cover.
\begin{lemma}
If $W$ is Weinstein and $\dim W=2n\geq 6$, then there is an adapted cover $pr_W:\tilde W \to W$.
\end{lemma}

\begin{proof}
The dimension assumption $2n\geq 6$ implies that $\pi_1(P)\cong \pi_1(W)$.
Hence there is also a corresponding subgroup of index $\ell$ in $\pi_1(W)$, so covering theory tells us that there is a $\ell$-fold cover $pr_W:\tilde W \to W$.
The cover $\tilde W$ inherits a Weinstein structure by pulling back the symplectic form, the Liouville vector field and the Morse function.
\end{proof}

On an adapted cover a fibered Dehn twist $\tau: W \to W$ can be lifted to a map $\tilde \tau: \tilde W \to \tilde W$.
Indeed, in a collar neighborhood of the boundary, $P\times I$, we see that $\tau$ lifts to the map
\[
\begin{split}
\tilde P\times I & \longrightarrow \tilde P\times I \\
(p,t) & \longmapsto (Fl^{R_{\tilde P}}_{f_m(t)/\ell}(p),t)
\end{split}
\]
This can be extended to the content of the page by using a deck transformation of the cover $\tilde W \to W$.
We will also call this extension $\tilde \tau$.

\begin{lemma}
The map $\tilde \tau$ is a symplectomorphism satisfying
$$
\tilde \tau^* \tilde \lambda=\tilde \lambda-d \tilde U,
$$
where $\tilde U=pr_W^*U$
\end{lemma}

\begin{definition}
We call the symplectomorphism $\tilde \tau$ a {\bf right-handed fractional fibered Dehn twist} of power $\ell$, or more briefly a {\bf fractional twist}.
If a Liouville domain $\tilde W$ is an adapted $\ell$-fold cover of $W$, then we say that $\tilde W$ admits a fractional Dehn twist of power $\ell$.
\end{definition}
Note that if a Liouville domain admits a fractional fibered Dehn twist of power $\ell$, then there is an action of $\Z_\ell$ by symplectomorphisms which induces rotation by roots of unity in a collar neighborhood of the boundary.
Write $\zeta_\ell$ for the symplectomorphism that generates this $\Z_\ell$-action on the content of the page, and which rotates the margin of the page by $2\pi/\ell$ as in the above construction. 

\begin{remark}
To remove some unnecessary clutter, we will omit the notation $\tilde{\phantom{W}}$ for the cover $\tilde W$ and $\tilde P$ when we only need the cover itself. 
\end{remark}

\begin{example}
\label{ex:std_dehn_twist}
Consider $(W=T^*_{\leq 1} \R \P^n,d\lambda_{can})$.
Since $\R \P^n$ admits a metric for which all geodesics are periodic with the same period, $W$ admits a fibered Dehn twist $\tau$. Its double cover $T^*S^n$ admits therefore a fractional fibered Dehn twist of power $2$.
This is the generalized Dehn twist that was already considered by Arnold, \cite{Arnold:monodromy_A1}.
\end{example}

\begin{example}
Consider the complex hypersurface of degree $d$ in $\C \P^n$ given by
$$
X_d=\{ [z_0:\ldots:z_n]~|~\sum_j z_j^d=0 \}
.
$$
Define $W_d=\C \P^n-\nu_{\C \P^n}(X_d)$.
This is a Weinstein domain with fundamental group $\pi_1(W_d)\cong \Z_d$.
It admits an adapted $d$-fold cover by the Brieskorn variety
$$
V_d=\{ (z_0,\ldots,z_n)\in \C^{n+1}~|~\sum_j z_j^d=1 \}
,
$$
see \cite[Section 7.3.2]{{CDvK:right-handed}}.
Hence $V_d$ admits a fractional fibered twist of power $d$.
\end{example}

\begin{remark}
\label{rem:conventions_twisting_profile}
Due to our conventions in the definition of a mapping torus, we will usually work with the inverse of the monodromy. This inverse also has a twisting profile, which we will denote by $f_i$; we have $f_i=-f_m$.
\end{remark}

\subsection{Preparing for filling obstructions}
We will use a symplectic field theory setup to describe filling obstructions.
In this setup one considers holomorphic curves in symplectic cobordisms.
To guarantee that this is a Fredholm problem, we need additional assumptions.
The simplest condition is the following.
\begin{definition}
A {periodic Reeb orbit} is called {\bf non-degenerate} if the restriction of its linearized return map to the contact structure has no eigenvalues equal to $1$.
We call a contact form {\bf non-degenerate} if all its periodic Reeb orbits are non-degenerate.
\end{definition}
Any contact form can be deformed into a non-degenerate one by a $C^\infty$-small perturbation. 
The Reeb dynamics can change dramatically under such a perturbation.
Since we are working with an $S^1$-symmetry, it is useful to also include the Morse-Bott setup.
\begin{definition}
A contact form $\alpha$ on $Y$ is said to be of {\bf
Morse--Bott type} if the following conditions hold.
\begin{itemize}
\item The action spectrum $\Spec (\alpha)$ is discrete.
\item For every $T\in \Spec (\alpha)$, the subset $N_{T}=\{ p\in Y|
Fl^{R_\alpha}_{T}(p)=p\}$ is a smooth, closed submanifold of $Y$ such
that the rank $d\alpha|_{N_{T}}$ is locally constant and
$T_pN_{T}=\ker (TFl^{R_\alpha}_{T}-id)_p$.
\end{itemize}
\end{definition}

We will also need suitable almost complex structures on the symplectization $\R \times Y$ of a contact manifold $Y$.
We first recall the notion of stable Hamiltonian structure, which we need in order to invoke symplectic field theory compactness, \cite{BEHWZ:compactness}.
\begin{definition}
A {\bf stable Hamiltonian structure} on $Y^{2n-1}$ is a pair $(\lambda,\Omega_{sH})$, where $\lambda$ is a $1$-form and $\Omega_{sH}$ a closed $2$-form such that
\begin{itemize}
\item{} $\ker \Omega_{sH} \subset \ker d\lambda$,
\item{} $\lambda \w \Omega_{sH}^{n-1}>0$.
\end{itemize}
\end{definition}
A cooriented contact structure $(Y,\xi=\ker \lambda)$ is stably Hamiltonian with respect to $(\lambda,d\lambda)$.
On the other hand, stable Hamiltonian structures do not necessarily come from contact structures.
Given a stable Hamiltonian structure, one can define a Reeb-like vector field by the equations
$$
i_{R_\lambda} \lambda=1,\quad i_{R_\lambda} \Omega_{sH}=0.
$$

We now discuss the appropriate class of almost complex structures for the symplectic field theory setup.

\begin{definition}
Let $Y$ be an oriented $(2n-1)$-dimensional manifold with stable Hamiltonian structure $(\lambda,\Omega_{sH})$.
An {\bf adjusted almost complex structure} $J$ on $\R \times Y$ is an endomorphism $J:T (\R \times Y) \to T (\R \times Y)$ such that
\begin{itemize}
\item{} $J^2=-\id$,
\item{} $J$ is $\R$-invariant,
\item{} $J \partial_t=R_\lambda$,
\item{} $J$ gives $\xi=\ker \lambda$ the structure of a complex vector bundle that is $\Omega_{sH}$-tame.
\end{itemize}
\end{definition}

\subsection{Moduli spaces of holomorphic curves and indices}
Before we give a filling obstruction, we briefly review some notions from holomorphic curve theory.
For simplicity, we describe the case of rational holomorphic curves in a symplectization.
This setup can be generalized to general symplectic cobordisms.

Fix a contact manifold $(Y^{2n-1},\xi=\ker \alpha)$, and let $\Sigma$ be a Riemann surface of genus $0$ with finitely many punctures, denoted by $\{ p_i \}_i\cup \{ q_j \}_j$.
We call $\{ p_i \}_i$ positive punctures and $\{ q_j \}_j$ negative punctures.
Suppose that $u:\Sigma\to \R \times Y$ is a holomorphic curve asymptotic to a collection of periodic Reeb orbits $\Gamma^+$ at the positive punctures and asymptotic to a collection of periodic Reeb orbits $\Gamma^-$ at the negative punctures.
Choose a trivialization $\Phi$ of $\xi|_{\Gamma^+\cup \Gamma^-}$. This allows us to define the Conley-Zehnder index of a periodic Reeb orbit.
Define the {\bf total Maslov index} of $u$ as
$$
\mu^{\Phi}(\Gamma^+;\Gamma^-)=
\sum_{\gamma \in \Gamma^+} \mu_{CZ}(\gamma,\Phi)
-
\sum_{\gamma \in \Gamma^-} \mu_{CZ}(\gamma,\Phi).
$$
This number depends on the trivialization $\Phi$, but whenever possible we will restrict ourselves to special trivializations, see Remark~\ref{rem:CZ_via_disk}, making this dependence irrelevant.
The Fredholm index of the linearized Cauchy-Riemann operator, used in Section~\ref{seq:MB_regularity}, is given by
\begin{equation}
\label{eq:ind_moduli}
\ind D_u=n \chi(\Sigma)+2 c_1^\Phi(u^*T(\R \times Y)\, )+\mu^{\Phi}(\Gamma^+;\Gamma^-)+ \# \Gamma^+ +\# \Gamma^-,
\end{equation}
where $c_1^\Phi(u^*T(\R \times Y)\, )$ is the relative Chern class of the trivialization $\Phi$, see \cite[Proposition~3.7]{Wendl:transversality}.
Observe that this index does not depend on the chosen trivializations.
Furthermore, this Fredholm index should be thought of as the ``virtual'' dimension of the space of holomorphic maps, as the following theorem, \cite{Dragnev}, shows. Let $\mathcal J$ denote the space of adjusted almost complex structures.
\begin{theorem}[Dragnev]
\label{thm:regularity_somewhere_injective}
There is a Baire set $\mathcal J_{reg}\subset \mathcal J$ such that for $J\in \mathcal J_{reg}$, every moduli space of simple $J$-holomorphic curves is regular.
In particular, the dimension of the space of simple holomorphic maps is given by the Fredholm index.
\end{theorem}
Although we will not use contact homology, it is useful to define the {\bf reduced index} of a non-degenerate periodic Reeb orbit $\gamma$ as 
$$
\bar \mu(\gamma;\Phi)=\mu_{CZ}(\gamma;\Phi)+n-3
$$
to see the relation with contact homology: the degree of a generator in contact homology is given by the reduced index.
Fix a relative homology class $A$ in $H_2(Y,\Gamma^+\cup \Gamma^-)$, and define $\mathcal Hol_0^A(\Gamma^+;\Gamma^-)$ as the space of holomorphic maps $u$ from $\Sigma$ to $\R \times Y$ that are asymptotic to the collection of periodic Reeb orbits $\Gamma^+$ near the positive punctures $\{ p_i \}_i$, asymptotic to the periodic orbits $\Gamma^-$ near the negative punctures such that $[u]=A\in H_2(Y,\Gamma^+\cup \Gamma^-)$.
Define the  {\bf moduli space of rational curves} with prescribed asymptotics and homology class by
$$
\mathcal M_0^A(\Gamma^+;\Gamma^-)=
\mathcal Hol_0^A(\Gamma^+;\Gamma^-)/\Aut \Sigma
.
$$
From now on, we will restrict ourselves to the case of a single positive puncture, so $\Gamma^+=\gamma^+$.

Suppose that $u$ represents an element $[u]\in \mathcal M_0^A(\gamma^+;\Gamma^-)$.
If we assume regularity, as for instance obtained by Dragnev's theorem, then the dimension of this moduli space is
\begin{equation}
\label{eq:dimension_formula}
\dim \mathcal M^A_0(\gamma^+;\Gamma^-)
=\bar \mu(\gamma^+, \Phi)-\sum_{\gamma^-_j \in \Gamma^-} \bar \mu(\gamma^-_j, \Phi)
+2 c_1^\Phi(u^*T(\R \times Y)\, )
.
\end{equation}
We say a curve $u$ representing $[u]$ in $\mathcal M^A_0(\gamma^+;\Gamma^-)$ is {\bf rigid} if $\dim \mathcal M_0(\gamma^+;\Gamma^-)=1$.
Note here that the symplectization has an $\R$-action, and modding out this action justifies the notion ``rigid''.

\begin{remark}
\label{rem:CZ_via_disk}
For practical purposes, we will consider special trivializations of the contact structure $(\xi,d\alpha)$.
Namely, if $\gamma$ bounds a disk $D$, then we can trivialize $(\xi,d\alpha)$ over $D$.
Denote this trivialization by $\Phi_D$ and we can define $\mu_{CZ}(\gamma,\Phi_D)$.
If we choose another bounding disk $D'$, formed as the connected sum $D'=D\# A$, where $A$ is sphere, then the Conley-Zehnder index changes according to the following formula,
$$
\mu_{CZ}(\gamma,\Phi_{D'})=\mu_{CZ}(\gamma,\Phi_D)+2\langle c_1(\xi),[A] \rangle .
$$ 
In particular, we see that this gives a well-defined total Maslov index if $c_1(\xi)=0$.
In this case, Formula~\eqref{eq:dimension_formula} simplifies since the relative Chern class will be $0$.
\end{remark}

\subsection{A filling obstruction}

We take a lemma from \cite{Massot:weak_strong_fillability} which we have slightly modified for our purposes.
\begin{lemma}[Massot, Niederkr\"uger, Wendl~+$\epsilon$]
\label{lemma:non-fillability}
Let $(Y^{2n-1},\xi=\ker \lambda)$ be a cooriented contact manifold with non-degenerate contact form $\lambda$, and a stable Hamiltonian structure $(\lambda,\Omega_{sH})$.

Suppose that there is an adjusted almost complex structure $J$ on the symplectization $\R \times Y$ with the following properties.
\begin{enumerate}
\item There is a Fredholm regular, rigid $J$-holomorphic finite energy plane $u_0$ in $\R \times Y$ that is asymptotic to a simply covered Reeb orbit $\gamma_1$.
\item If $u$ is a finite energy $J$-holomorphic curve of genus-$0$ with a single positive puncture, at which it is asymptotic to $\gamma_1$, then $u$ is a translation of $u_0$, or
$[u]\in \mathcal M_0(\gamma_1;\gamma_0)$, where $\gamma_0$ is a periodic Reeb orbit that satisfies $\mathcal A(\gamma_0)<2\min_\gamma \mathcal A(\gamma)$ and $\dim \mathcal M_0(\gamma_1;\gamma_0)>1$.
\end{enumerate}
Then $(Y,\xi)$ does not admit any semi-positive weak filling $(F_0,\Omega)$ for which $\Omega|_{TY}$ is cohomologous to $\Omega_{sH}$.
Furthermore, if $\Omega_{sH}=d\lambda$ then every contact form for $(Y,\xi)$ admits a contractible periodic Reeb orbit.
\end{lemma}

\begin{remark}
The second condition is only slightly more general than the one in \cite{Massot:weak_strong_fillability}. Furthermore, this condition is rather artificial, but we will need it to cover the case of $T^*\H \P^m$ and $T^*Ca\P^2$ in Theorem~\ref{thm:result}.
The statement in \cite{Massot:weak_strong_fillability} suffices for all other cases.
\end{remark}
For completeness, we include a proof, which is almost the same as the one in \cite{Massot:weak_strong_fillability}.
\begin{proof}
Suppose there is a weak filling $(F_0,\Omega)$ for $(Y,\xi)$ with $[\Omega|_{Y}]=[\Omega_{sH}]$.
According to \cite[Lemma 2.10]{Massot:weak_strong_fillability} there is a cylindrical end $([0,\infty[\times Y,\Omega)$ with the properties that
\begin{itemize}
\item there is $T>0$ with $\Omega=\Omega_{sH}+d(e^t\lambda)$ on $[T,\infty[\times Y$.
\item on $[0,\epsilon[\times Y$, $\Omega$ restricts to the given symplectic form on $F_0$.
\end{itemize} 

Attach this cylindrical end along a collar neighborhood of the boundary of $F_0$ to form a complete symplectic manifold, which we denote by $F$.
\begin{remark}
For the case of a strong filling (so $\Omega=d\lambda$), which is the only case we really need, the above argument can be simplified: we can attach the positive part of the symplectization as a suitable cylindrical end.
\end{remark}

Extend the adjusted almost complex structure $J$ for $[0,\infty[\times Y\subset \R \times Y$ given in the assumptions to an almost complex structure on $F$ taming $\Omega$.
By a result due to Dragnev \cite{Dragnev}, stated in Theorem~\ref{thm:regularity_somewhere_injective}, we can assume that all simple $J$-holomorphic curves in $F$ and in the symplectization $\R \times Y$ are regular.

The holomorphic curve $u_0$ has image in $[T,\infty [ \times Y \subset F$ for some $T$, so $u_0$ represents an element $[u_0]\in \mathcal M_0(\gamma_1;\emptyset)$, the moduli space of holomorphic finite energy planes asymptotic to $\gamma_1$.
Furthermore, since all simple holomorphic curves are regular by our choice of $J$, the dimension of the component of $\mathcal M_0(\gamma_1;\emptyset)$ containing $[u_0]$ can be extracted from the Fredholm index.
We denote this component by $\mathcal M$.
It is a smooth manifold of dimension $\dim \mathcal M=1$.

Take a sequence of holomorphic planes $\{ u_k \}$ with $[u_k]\in \mathcal M$.
By SFT compactness, \cite{BEHWZ:compactness}, there is a subsequence converging to a holomorphic building with levels $u_\infty=\{ u_\infty^{L_1},\ldots,u_\infty^{L_m} \}$.
Take the first non-trivial level, say $u_\infty^{L_j}$.
This is a curve that is asymptotic to $\gamma_1$.
By the assumptions we have made, only two cases can occur, namely
\begin{enumerate}
\item $u_\infty^{L_j}$ is a translation of $u_0$.
\item $u_\infty^{L_j}$ is a holomorphic cylinder from $\gamma_1$ to $\gamma_0$.
\end{enumerate}
We first argue that the second case cannot occur.
Indeed, the total building must be a plane with possibly sphere bubbles, so there must be a holomorphic building with the topological type of a plane capping off $\gamma_0$.
The assumptions on the action of $\gamma_0$ tell us that $\gamma_0$ and all possible periodic Reeb orbits appearing in a building capping off $\gamma_0$ must be simple, so all components (not considering sphere bubbles) must be somewhere injective.
Furthermore, the index of $\gamma_0$ is lower than that of $\gamma_1$, so it follows that the plane capping off the final periodic Reeb orbit cannot exist by a regularity argument using Theorem~\ref{thm:regularity_somewhere_injective}: use that the index of such a final plane is negative, and that its asymptote is embedded.

We conclude that the first case occurs, so we have at most one level, the so-called main layer.
In this main layer, we hence have at most sphere bubbles, and we can write the limit curve as $u_\infty \cup \cup_i B_i$.

\begin{figure}[htp]
\def\svgwidth{0.5\textwidth}%
\begingroup\endlinechar=-1
\resizebox{0.5\textwidth}{!}{%
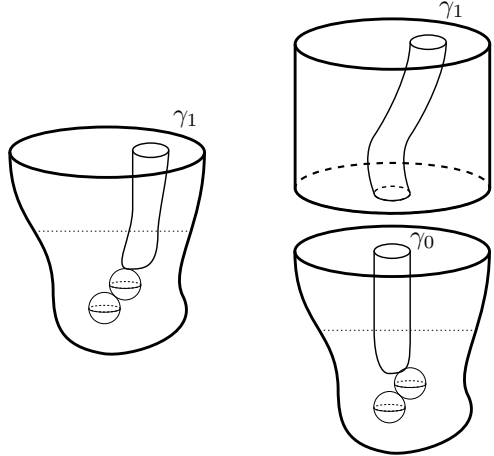%
}\endgroup
\caption{Possible breaking in the filling}
\label{fig:breaking}
\end{figure}

We claim that semi-positivity excludes these sphere bubbles, and we give a brief sketch of the argument.
The Fredholm index is additive, so in a family we get
$$
\ind u_0=\ind u_\infty +\sum_i \ind B_i,
$$
and we have $\ind u_0=\ind u_\infty$ by the above argument.
We hence conclude that $\sum_i \ind B_i=0$.
On the other hand, the index of a holomorphic sphere in $F^{2n}$ is given by
$$
\ind B_i=2n+2\langle c_1(F),B_i \rangle.
$$
We find that there is $i_0$ with $\langle c_1(F),B_{i_0} \rangle <0$.
Furthermore, $B_{i_0}$ can be written as $B_{i_0}=k_{i_0} A_{i_0}$, where $A_{i_0}$ is a simple holomorphic curve.
It follows that  $\langle c_1(F),A_{i_0} \rangle <0$ and this contradicts Lemma~6.4.4 from \cite{MS:J_curves}.

The statement concerning the Weinstein conjecture can be found in \cite[Lemma 3.3]{Massot:weak_strong_fillability}.
The basic idea is due to Hofer: stretching a finite energy plane in a symplectization gives rise to a holomorphic building which topologically still forms a plane. Hence the lowest level contains a finite energy plane, which is asymptotic to a contractible periodic Reeb orbit.
\end{proof}

\section{$S^1$-invariant contact structures, and geometric differences between positive and negative twists}
\label{sec:invariant_contact}
Let $Y\to M$ be a principal circle bundle and denote the right principal action of $g\in S^1$ by $RA_g$.
Define the vector field $X_Y$ generating the $S^1$-action by
$$
X_Y:=\frac{d}{dt}|_{t=0} RA_{e^{it}}.
$$
Suppose that $\xi$ is a cooriented, $S^1$-invariant contact structure on $Y$, meaning
$$
(RA_g)_*\xi=\xi.
$$
By averaging the contact form, we obtain an $S^1$-invariant contact form $\alpha$, so $\mathcal L_{X_Y}\alpha=0$.
Recall the following notion introduced by Giroux, \cite{Giroux:convex}.
\begin{definition}
A hypersurface $H$ in a contact manifold $(Y,\xi)$ is called {\bf convex} if there exists a contact vector field $X$ that is transverse to $H$.
The {\bf dividing set} of $H$ with respect to $X$ is the set
$$
\Gamma:=\{ x \in H ~|~X(x)\in \xi \} .
$$
\end{definition}

Let $\pi:Y\to M$ be a principal circle bundle with an invariant contact form $\alpha$.
Define the {\bf almost dividing set} of $(Y,\alpha)$ with respect to $M$ as the set
$$
\Gamma:=\{ x\in M ~|~\alpha_q(X_Y)=0 \text{ for }q\in \pi^{-1}(x) \} .
$$
This is well-defined, since $\alpha$ is $S^1$-invariant.
Now let $Z$ be a subset of $M$ such that $Y|_{M-Z}$ is a trivial bundle, so we can find a section $\sigma:M-Z \to Y$.
Then $\sigma(M-Z)$ is a convex surface with respect to $X_Y$.
Furthermore, $\pi^{-1}(\Gamma)\cap \sigma(M-Z)$ is the dividing set of $\sigma(M-Z)$ with respect to $X_Y$.

According to the following proposition the almost dividing set is a contact manifold.
\begin{proposition}
Let $(Y,\xi=\ker \alpha_{inv})$ be a principal circle bundle over $M$ with an invariant contact structure. Suppose that $\Gamma$ is the almost dividing set of $(Y,\alpha_{inv})$ in $M$.
Then $\alpha_{inv}$ induces a contact structure on $\Gamma$.
\end{proposition}
A proof of this statement can be found in \cite[Lemma~3.3]{DingGeiges:circle_bdls}.
Alternatively, one can prove this statement using contact reduction.

We shall construct invariant contact structures in higher dimensions that should be considered overtwisted in view of the $3$-dimensional analogue. 
To compare with dimension $3$, the following result due to Giroux, \cite[Proposition 4.1]{Giroux:circle_bdls}, is closest to what we will obtain.
\begin{proposition}[Giroux]
Suppose $(Y^3,\alpha)$ is a principal circle bundle over a surface $S$ with almost dividing set $\Gamma$.
Then the contact manifold $(Y,\alpha)$ is universally tight if and only if one of the following conditions is satisfied
\begin{itemize}
\item $S\neq S^2$, and no component of $S-\Gamma$ is a disk.
\item $S=S^2$, the Euler class of $Y$ is negative and $\Gamma$ is empty.
\item $S=S^2$, the Euler class of $Y$ is non-negative, and $\Gamma$ consists of a single circle.
\end{itemize}
\end{proposition}
After rescaling the contact form, the second case corresponds to prequantization bundles over $S^2$.

\begin{remark}[Goal of the paper]
We shall construct a principal circle bundle $Y$ over higher-dimensional manifolds $M$ with the property that the almost dividing set bounds disk-bundles over codimension $2$ submanifolds in $M$.
In many cases the almost dividing set is actually connected.

In dimension $3$ such manifolds can be universally tight, tight, or overtwisted depending on more precise data.
Analogously, we shall obtain manifolds that can be symplectically fillable or non-fillable, depending on more precise data.
\end{remark}

\subsection{Circle actions}
Let $(P,\lambda_P)$ be a prequantization bundle over an integral symplectic manifold $(Q,\omega)$.
Denote the Reeb field of $\lambda_P$ by $R_{P}$.
This vector field generates the circle action on $P$.
Fix integers $a,b$ and consider the circle action on $P\times D^2$ given by
\[
\begin{split}
S^1 \times P\times D^2 & \longrightarrow P \times D^2 \\
(g;p,z) & \longmapsto (g^a \cdot p,g^b z).
\end{split}
\]
The vector field generating this action is
$$
X_{ab}=a R_{P} +b\partial_\phi.
$$
As in the usual model for a contact form near the binding of an open book, choose functions $h_1$ and $h_2$.
We shall only impose the contact condition, $h_1 h_2'-h_2 h_1'\neq 0$, though.

\begin{lemma}
\label{lemma:invariant_form_binding}
The contact form $\alpha_{inv}=h_1(r)\lambda_P+h_2(r)d\phi$ is invariant under the $S^1$-action.
\end{lemma}
This can by verified by computing the Lie derivative with the Cartan formula,
$$
\mathcal L_{X_{ab}} \alpha_{inv}=d\left( ah_1+bh_2 \right)-ad h_1-bdh_2=0.
$$
Choose a smooth function $f_i$, which is going to serve as a profile function for the inverse of a fractional twist.
We do not yet impose any conditions.
Define an equivalence relation on $P\times I \times \R$
$$
(p,t;\phi) \sim (Fl^{R_{P}}_{f_i(t)}(p),t;\phi+U_i(t)\, ),
$$
where $U_i$ is defined via \eqref{eq:formula_u} using $f_i$ instead of $f_m$.
Since we haven't imposed any conditions on the profile function $f_i$, the following construction will not be a contact open book in the sense we defined it in Section~\ref{sec:contact_OB}.
However, the resulting contact manifold can be deformed into a contact open book. We will explain this in the proof of Theorem~\ref{thm:left_and_right_twisted}.

Define the ``margins of the pages'' of an ``open book'' by $P\times I\times \R/\sim$.
The set $P\times D^2$ serves as a neighborhood of the binding of the ``open book'' to which we glue to the ``pages'' using the gluing map
\begin{equation}
\label{eq:gluingmap_binding_to_pages}
\begin{split}
\psi_G: P\times D_{glue} & \longrightarrow P\times I \times \R /\sim \\
(p;r,\phi) & \longmapsto ( [\frac{f_i(1-r)\phi}{2\pi}] \cdot p,1-r,\frac{\phi}{2\pi}U_i(1-r)\, ).
\end{split}
\end{equation}
Here $D_{glue}$ is the open annulus $D_{glue}=\{ r e^{i\phi} \in D^2~|~ r_1-\delta_{glue}<r <r_1 +\delta_{glue} \}$, whose inner and outer radius will be specified later.
We call the domain $ P\times D_{glue}$ the {\bf gluing region}.

For now, we will compute with the gluing region $P\times (D^2 -\{ 0 \})$
\begin{lemma}
The map $\psi_G$ induces a circle action on $P\times I \times \R /\sim$ generated by the vector field
\begin{equation}
\label{eq:vf_S^1_pages}
\tilde X_{a,b}=\frac{2\pi a+f_i(t)b}{2\pi} R_{P}+\frac{bU_i(t)}{2\pi}\partial_\phi .
\end{equation}
Furthermore, the contact structure $\ker (d\phi+e^t \lambda_P)$ is $S^1$-invariant.
\end{lemma}

\begin{proof}
Take $g\in S^1$.
On $P\times D^2$, the element $g$ induces the map
\begin{equation}
\label{eq:action_binding}
(p;z) \longmapsto (g^a \cdot p,g^b z).
\end{equation}
With the map $\psi_G$, we get
$$
\psi_G(\, (g^a \cdot p,g^b z) \, )=([ag+\frac{bg}{2\pi}f_i(1-r)+\frac{\phi}{2\pi}f_i(1-r)]\cdot p,1-r,\frac{\phi}{2\pi}U_i(1-r)+ \frac{bg}{2\pi}U_i(1-r)\,).
$$
The circle action on $P\times I \times \R /\sim$ is hence given by
\[
\begin{split}
S^1 \times \left(  P\times I \times \R /\sim \right) & \longrightarrow  P\times I \times \R /\sim \\
(g;p,t,\phi) & \longmapsto ( [\frac{2\pi a+bf_i(t)}{2\pi}g] \cdot p,t,\phi+\frac{bg}{2\pi}U_i(t) ).
\end{split}
\]
On the right hand side we have used additive notation rather than multiplicative.
If we denote the vector field generating the circle action by $\tilde X_{a,b}$, then we get
$$
\tilde X_{a,b}=\frac{2\pi a+f_i(t)b}{2\pi} R_{P}+\frac{bU_i(t)}{2\pi}\partial_\phi .
$$

To check that the contact form $\left( d\phi+e^t \lambda_P \right)$ is invariant under the circle action, we compute the Lie derivative with respect to $\tilde X_{a,b}$ with the Cartan formula and simplify with Formula~\eqref{eq:formula_u},
$$
\mathcal L_{\tilde X_{a,b}} \left( d\phi+e^t \lambda_P \right) =
d\left( \frac{bU_i(t)}{2\pi} \right) +d\left( \frac{2\pi a+f_i(t)b}{2\pi}e^t \right)
-\frac{2\pi a+f_i(t)b}{2\pi}e^t dt=0.
$$
\end{proof}

Now suppose that $P$ is symplectically filled by a Liouville domain $(W,\Omega=d\lambda)$, where $\lambda=e^t\lambda_P$ in a collar neighborhood of the boundary $P\times I$.
Define the manifold
$$
Y:=P\times D^2 \cup_{\psi_G} W\times \R /\sim.
$$
Define the content of the page $W$ by $W_{in}:=W-P\times I$.
In order to obtain an invariant contact structure we need to fulfill the following criteria:
\begin{itemize}
\item the monodromy extends to $W_{in} \times \R /\sim$.
This will be done with a fractional twist of power $\ell$ ($\ell$ is some positive integer); let $\zeta_\ell=e^{2\pi i/\ell}$ denote a root of unity. Define the equivalence relation
$$
(w,\phi) \sim(\zeta_\ell \cdot w,\phi+U_i).
$$
\item the action extends to $W_{in} \times \R /\sim$.
As we shall see in the following lemma, a sufficient condition to do this is $2\pi a+f_i(0) b=0$.
\end{itemize}

\subsubsection{Right- and left-handed twists}
\label{sec:right_left_binding_twist}
We briefly describe the \emph{monodromy} that we shall be using.
We use the special twisting profiles $f_m$ as sketched in Figure~\ref{fig:profile_function}.
For a right-handed fractional twist of power $\ell$, we choose $f_m:[0,1]\to \R$ to be a positive, smooth function with
\begin{itemize}
\item $f_m\equiv 2\pi/\ell$ on a neighborhood of $0$, say $[0,\delta_c]$.
\item for $t>1-\delta_c$, the function $f_m(t)$ is small, and strictly increasing with small slope $f_m'$, say $|f_m'|<\delta_s$.
\end{itemize}
For a left-handed fractional fibered twist of power $\ell$, we choose $f_m:[0,1]\to \R$ to be an increasing, smooth function with
\begin{itemize}
\item $f_m\equiv -2\pi/\ell$ on a neighborhood of $0$, say $[0,\delta_c]$.
\item for $t>\delta_c$, the function $f_m(t)$ is strictly increasing. Furthermore, it has a unique zero in $t=t_1$.
For $t>t_1$, the slope $f_m'$ is small, say $|f_m'|<\delta_s$.
\end{itemize}
\begin{figure}[htp]
\def\svgwidth{0.65\textwidth}%
\begingroup\endlinechar=-1
\resizebox{0.65\textwidth}{!}{%
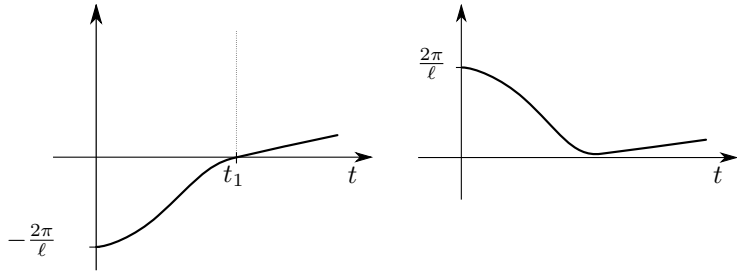%
}\endgroup
\caption{Profile functions for a left-handed fractional twist (on the left) and a right-handed fractional twist (on the right)}
\label{fig:profile_function}
\end{figure}
With these choices of twisting profiles, fractional twists are not the identity near the boundary.
However, we shall see that this choice also glues nicely to a neighborhood of the binding.

\subsubsection{Gluing the binding to the pages}
We will make some choices to fix the gluing region. The eventual results do not depend on these choices up to contactomorphism, but a small gluing region will be convenient.
If $a=-1$ (the left-handed twist), then choose $r_1=1-t_1$.
If $a=+1$ (the right-handed twist), then choose $r_1=1-t_m$, where $t_m$ attains the minimal value of $f_m$.
In both cases, choose $\delta_{glue}$ small.

\begin{lemma}
\label{lemma:decomposition} 
Suppose that $W$ admits a fractional twist of power $\ell$, and denote the boundary by $P$.
Take $a=\pm 1$ and $b=\ell$.
Choose a smooth profile function $f_i$ (for the \emph{inverse} of the monodromy) such that $f_i(0)=-\frac{a 2\pi}{\ell}$, and $f_i(1-r_1+\delta_{glue}+t)=-\left( \delta_1+\delta_2 t \right)$, where $\delta_1,\delta_2>0$ are small.
Then the manifold
$$
Y:=P\times D^2 \cup_{\psi_G} W \times \R /\sim
$$
is a principal circle bundle with an $S^1$-invariant contact form $\alpha_{inv}$.

The quotient space $M:=Y/S^1$ can be identified with the smooth manifold,
$$
M\cong P\times_{S^1} D^2 \cup_{\bar \psi_G} W/\Z_\ell,
$$
where $P\times_{S^1} D^2$ is the associated disk-bundle for the action given by \eqref{eq:action_binding}, and $\bar \psi_G$ is the map induced by $\psi_G$.
Furthermore, if $a=-1$, the almost dividing set is nonempty and contactomorphic to $P/\Z_\ell$. 
\end{lemma}

Since the gluing map in our construction comes from the inverse of the monodromy, the profile function $f_i$ in the above lemma is minus the profile function $f_m$ used for a fractional twist.
See Figure~\ref{fig:profile_function} for the twisting profiles for the monodromy we will be using.

\begin{proof}
We only carry out the proof for the case most interesting to us in this paper: $a = -1$. The case $a=1$ and $b=1$ was done in detail in our previous paper \cite[Section 6]{CDvK:right-handed}.
Alternatively, one can adapt the argument here.
  
We first show that $P\times D^2 \cup_{\psi_G} W \times \R /\sim$ admits a free $S^1$-action.
We define the action on subsets and show that it is well-defined.
On $P\times D^2$ we have the action
\[
\begin{split}
S^1 \times P \times D^2 & \longrightarrow P \times D^2\\
g\cdot (p,z) &\longmapsto \left(
g^{-1}\cdot p,g^\ell\cdot z
\right)
.
\end{split}
\]
On $P\times I\times \R /\sim$ we have the action
\[
\begin{split}
S^1 \times P \times I \times \R /\sim & \longrightarrow P \times I \times \R /\sim\\
g\cdot (p,t,\phi) &=\left(
\frac{\ell f_i(t)-2\pi}{2\pi}g \cdot p,t,\phi+\frac{\ell g}{2\pi}U_i(t)
\right)
\end{split}
\]
For the action on $W_{in}\times \R /\sim$ note first of all that the function $U_i$ is constant on that set. Define the circle action by
\[
\begin{split}
S^1 \times \left( W_{in} \times \R /\sim  \right) & \longrightarrow  W_{in} \times \R /\sim\\
g\cdot [w,\phi] &\longmapsto [
w,\phi+\frac{\ell g}{2\pi}U_i
]
\end{split}
\]
Observe that on the last piece we have
$$
(w,\phi+\ell U_i)=(\zeta_\ell^\ell\cdot w,\phi+\ell U_i)\sim(w,\phi),
$$
so this action is an honest $S^1$-action.
By our assumptions on $f_i$, the actions on the overlap of the different pieces coincide, so we have a well-defined action.

We check that the action is free.
\begin{itemize}
\item on the set $P\times D^2$, this is clear since $S^1$ acts freely on $P$.
\item on the set $P\times I \times \R /\sim$, it suffices to check for $g\in [0,2\pi[$ that $g\cdot [p,t,\phi]=[p,t,\phi]$ if and only if $g=0$.
To see this holds, take $g$ such that $g\cdot [p,t,\phi]=[p,t,\phi]$, and  note that $\frac{\ell g}{2\pi}U_i(t)$ must be an integer multiple of $U_i(t)$, so $g=\frac{2\pi}{\ell}m$ for some $m \in \Z$.
For the $P$-factor, we must have
$$
\frac{\ell f_i(t)-2\pi}{2\pi}g \equiv m f_i(t) \mod 2\pi,
$$
so we get the condition
$$
(f_i(t)-\frac{2\pi}{\ell}) m \equiv f_i(t) m \mod 2\pi.
$$
Hence $m\in \ell\Z$. This implies $g\in 2\pi \Z$.
\item on the set $W_{in}\times \R /\sim$, the circle action is given by
$$
g\cdot [x,\phi]=[x,\phi+\frac{\ell g}{2\pi}U_i].
$$
The equivalence class $[x,\phi]$ contains the elements $\{ \zeta_\ell^{m} \cdot x,\phi+m U_i \}_{m \in \Z}$, so we see that we can only have $g\cdot [x,\phi]=[x,\phi]$ if $g \in 2\pi \Z$.
It follows that the action is free.
\end{itemize}

For the assertion about the quotient space $M$ we check that
$$
(W\times \R /\sim)/S^1\cong W /\Z_\ell.
$$
To see this, note that we can use the circle action to bring an element into the form $[x,0]$. There are $\ell$ such elements, namely $[\zeta_\ell^{m} \cdot x,0]$ for $m=0,\ldots,\ell-1$. Hence we need to mod out by this $\Z_\ell$-action.

To obtain an invariant contact structure, we use the contact form
$$
\alpha=d \phi+\lambda
$$
on the set $W \times \R /\sim$.
By the previous lemma, this gives an invariant contact structure.
We pull this form back to a neighborhood of the binding to see what behavior we need to prescribe there.
Write $Inv(r)=1-r$.
By the Cartan formula $\mathcal L_{\frac{f_i\phi}{2\pi} R_{P} } e^t \lambda_P=d(e^t\frac{f_i\phi}{2\pi})-e^t\frac{f_i\phi}{2\pi} dt$, so we find
\begin{equation}
\label{eq:pullback_form}
\begin{split}
\psi_G^* \alpha & =d\left( \frac{\phi U_i\circ Inv(r)}{2\pi} \right) + \psi_G^*e^t \lambda_{P}\\
&=Inv^*
\left(
\frac{U_i}{2\pi}d\phi+\frac{\phi}{2\pi}dU_i+\frac{e^tf_i(t)}{2\pi}d\phi+\frac{e^t df_i}{2\pi} \phi+e^t \lambda_{P}
\right) \\
&=e^{1-r} \lambda_P+ \frac{ \tilde C+\int_0^{1-r} e^s f_i(s) ds}{2\pi}d\phi.
\end{split}
\end{equation}
Put $h_1(r)=e^{1-r}$ and $h_2(r)=\frac{ \tilde C+\int_0^{1-r} e^s f_i(s) ds}{2\pi}$ for $r>\delta$, where $\delta$ is some positive number, to see that this form coincides with the one given in Lemma~\ref{lemma:invariant_form_binding}.
Indeed, observe that $h_2'(r)=\frac{-1}{2\pi}e^{1-r}f_i(1-r)$ is positive for small $r$.
Therefore we can extend it to the whole set $P\times D^2$ as an invariant contact form that is of open book type near $r=0$.

To obtain the claim about the almost dividing set, observe that
$$
i_{X_{-1,\ell}}\alpha_{inv}=\ell h_2(r)-h_1(r).
$$
For $r=r_1+\delta_{glue}$, this is positive, and for $r=0$, it is negative.
Since $\ell h_2'(r)-h_1'(r)>0$ on the interval $[0,r_1+\delta_{glue}[$, it follows that there is a unique $r_0\in [0,r_1+\delta_{glue}[$, where $i_{X_{-1,\ell}}\alpha$ vanishes.
If we go further into the page, we use the page model $W\times \R /\sim$; insert the vector field generating the $S^1$-action, \eqref{eq:vf_S^1_pages} and its extension to the content of the pages, into the contact form.
We get a function that is non-decreasing as we go deeper into the page (decreasing $t$-coordinate).
In particular, we see that the $S^1$-action is positively transverse to the contact structure on the content of the pages.
The almost dividing set is hence the set $(P\times \{ r_0 \}\times S^1)/S^1\cong P/\Z_\ell$ (see Formula~\eqref{eq:action_binding} with $a=-1,b=\ell$).
\end{proof}

\begin{remark}
Doing the construction of the above proof in the case of $a=1$ will not give the standard Boothby--Wang form.
Instead, we only obtain an invariant contact structure with a positively transverse contact vector field.
By rescaling $\alpha_{inv}$ we obtain the standard Boothby--Wang contact form.
\end{remark}

In Figure~\ref{fig:circle_actions} we have visualized the difference between the circle actions on an open book with a left-handed twist and one with a right-handed twist.

\begin{figure}[htp]
\def\svgwidth{0.85\textwidth}%
\begingroup\endlinechar=-1
\resizebox{0.85\textwidth}{!}{%
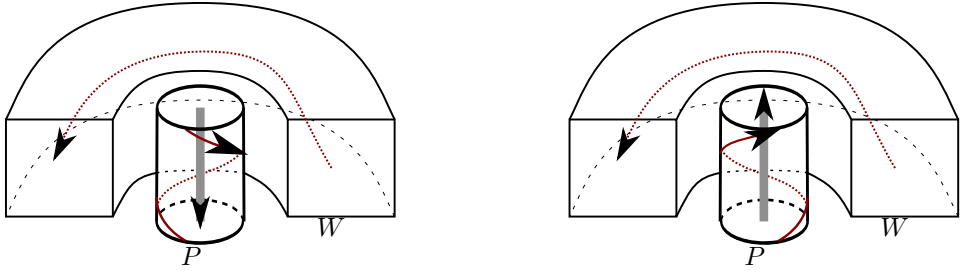%
}\endgroup
\caption{Direction and orbits of the circle action on an open book with a left-handed twist (left) and a right-handed twist (right)}
\label{fig:circle_actions}
\end{figure}

\begin{theorem}
\label{thm:left_and_right_twisted}
Let $W$ be a Liouville domain with boundary $P$ admitting a right-handed fractional twist $\tau$ of power $\ell$.
Then we have
\begin{itemize}
\item[(R)] the contact open book $\OB(W,\tau)$ is contactomorphic to a prequantization bundle $(Y,\alpha)$ over a symplectic manifold. In particular, the almost dividing set is empty, and the contact manifold is convex fillable.
\item[(L)] the contact open book $\OB(W,\tau^{-1})$ is diffeomorphic to a principal circle bundle over a smooth manifold $M$, and supports an $S^1$-invariant contact form $\alpha_{inv}$. Furthermore, the almost dividing set of $\alpha_{inv}$ is contactomorphic to $P/\Z_\ell$.
\end{itemize}
\end{theorem}

\begin{proof}
We apply Lemma~\ref{lemma:decomposition} which does not directly give a contact open book.
However, we can deform the contact form $\alpha_{inv}$ to a contact form of open book type: choose a $1$-parameter family of functions $h_2^s$ such that $h_2^0=h_2$, $h_2^1$ is constant on some open interval away from $r=0$, and the contact condition, $h_1{h_2^s}'-h_2^sh_1'\neq 0$, holds for all $s$.
By Gray stability, the resulting contact manifolds are contactomorphic.

The open interval contains some point $r_c$, so by taking a suitable neighborhood $P\times D^2_{r_c}$, we find the decomposition 
$$
Y=P\times D^2_{r_c} \cup_{\psi_G} W\times \R /\sim.
$$
By construction, the monodromy is isotopic to a fractional twist.
This proves most of the assertions, except for the statement about convex fillability. We prove the last claim in Lemma~\ref{lemma:fillability_BW}.
\end{proof}

\subsection{Examples}
The left-handed stabilization of the standard contact sphere $\OB(D^{2n},\id)$ is given by the contact open book
$$
\OB(T^*S^{n},\tau^{-1}),
$$
where $\tau$ is a right-handed Dehn twist. By Example~\ref{ex:std_dehn_twist}, this is a special case of a fractional fibered Dehn twist.

\begin{proposition}
\label{prop:negative_stabilization_as_invariant_contact}
The left-handed stabilization $\OB(T^*S^{n},\tau^{-1})$ is diffeomorphic to the Hopf fibration over $\C \P^n$. However, it has an almost dividing set contactomorphic to $(ST^*\R \P^{n},\lambda_{can})$.
\end{proposition}

\begin{proof}
According to Lemma \ref{lemma:decomposition}, we have the decompositions
\[
M\cong ST^*S^n\times_{S^1} D^2 \cup T^*S^n/\Z_2.
\] 
Now note that $ST^*S^n\times_{S^1} D^2 \cong \mathcal{O}_{Q^{n-1}}(-2)$ is the line bundle dual to the neighborhood of the quadric in $\C \P^n$ and $T^*S^n/\Z_2 \cong T^*\R \P^n$. Gluing these two pieces we conclude that $M$ is diffeomorphic to $\C \P^n$, and the boundary of the disk bundle, $ST^*\R \P^{n}$, is the almost dividing set.   
\end{proof}

\begin{remark}
The right-handed stabilization of the standard contact sphere $\OB(D^{2n},\id)$ is given by the contact open book $\OB(T^*S^{n},\tau)$. It is contactomorphic to the Hopf fibration over $\C \P^n$. Its almost dividing set is empty.  
\end{remark}

\begin{example}
We can also consider a right-handed fibered Dehn twist $\tau$ on $D^{2n}$.
In this case $\tau$ is symplectically isotopic to identity relative to the boundary.
Hence
$$
\OB(D^{2n},\tau)\cong \OB(D^{2n},\id)\cong (S^{2n+1},\xi_0).
$$
Furthermore, the $S^1$-invariant contact structure from Lemma~\ref{lemma:decomposition} is the standard prequantization structure, so the almost dividing set in $M=(S^{2n-1}\times_{S^1} D^2) \cup D^{2n}\cong \C \P^n$ is empty.

On the other hand, we can also consider a left-handed fibered twist, which is also symplectically isotopic to the identity relative to the boundary.
Then
$$
\OB(D^{2n},\tau^{-1})\cong \OB(D^{2n},\id)\cong (S^{2n+1},\xi_0).
$$
In this case, the $S^1$-invariant contact structure from Lemma~\ref{lemma:decomposition} has an almost dividing set contactomorphic to $(S^{2n-1},\xi_0)$ in $\overline {\C \P}^n$.
\end{example}

\section{Reeb orbits, Maslov indices and actions}
\label{sec:indices}
Let $(W^{2n-2},d\lambda)$ be a Liouville domain with boundary $P=\partial W$. 
Write $\lambda_P:=\lambda|_{P}$ for the contact form on $P$.
Assume that $(P,\lambda_P)$ is a prequantization bundle over an integral symplectic manifold $(Q,k\omega)$, where $\omega$ is primitive and $k\in \Z_{\geq 1}$.
Let $\psi$ be a symplectomorphism on $W$ that is the identity near the boundary.
Consider the contact open book $Y=\OB(W,\psi)$.  
The contact form in a neighborhood of the binding, $P\times D^2$, is given by
$$
\alpha=h_1(r) \lambda_P+h_2(r)d\phi.
$$
Define the matrix
$$
H=\left(
\begin{array}{cc}
h_1 & h_2 \\
h_1' & h_2'
\end{array}
\right)
.
$$
Then the Reeb vector field is given by
$$
R_\alpha=\frac{1}{\det H}\left( h_2' R_{P} -h_1' \partial_\phi \right).
$$
Reeb orbits have hence the form $t\mapsto (\gamma_P(\frac{h_2'(r)}{\det H(r)}t),r,-\frac{h_1'(r)}{\det H(r)}t)$, where $\gamma_P$ is a Reeb orbit in the binding.
We see that a Reeb orbit in $P\times D^2$ is periodic if and only if $\frac{h_2'(r)}{h_1'(r)}\in \Q$ because the Reeb flow $R_P$ is periodic with period $2\pi$.
Hence we can parametrize a simple periodic Reeb orbit by
$$
\gamma_{i,j,\phi_0}(t)=(\gamma_P(jt),r,i t),
$$
where $i,j$ are relatively prime and satisfy $-\frac{h_2'(r)}{h_1'(r)}=\frac{j}{i}$.
The action of these simple orbits is given by
\begin{equation}
\label{eq:action_near_binding}
\mathcal A(\gamma_{i,j,\phi_0})=2\pi h_1(r) j +2\pi h_2(r) i.
\end{equation}
\subsubsection{Reeb dynamics away from the binding}

In a general contact open book, one can understand periodic Reeb orbits that lie in the pages in terms of fixed points of iterates of the monodromy.
Indeed, if $\psi^m(x)=x$ for some $x$ in $W$, then there is $T>0$ such that $Fl^{R}_T(x,\phi)=(x,\phi)$: the Reeb field on the pages is given by $R=\partial_\phi$, so each turn around the binding corresponds to an application of $\psi$.
\begin{equation}
\label{eq:identification}
Fl^R_t(x,\phi)=(x,\phi+t)\sim (\psi^{-1}(x),\phi+t+U_i(x)\,).
\end{equation}
Hence we have
\begin{lemma}
Periodic Reeb orbits that lie in the pages are in $1-1$-correspondence to fixed points of the monodromy $\psi$ and its iterates.
\end{lemma}

Note that for the class of monodromies we will be considering from now on, namely left-handed fractional twists, the fixed point sets come in families, and most of the interesting behavior happens in a (large) neighborhood of the binding. Topologically, this set has the form $P\times D^2$.

Define $\Sigma_{i,j}$ to be the set of points $x$ in the margin part of the pages of the contact open book $Y=\OB(W,\psi)$ such that $x$ lies on a periodic Reeb orbit $\gamma_x$ with the properties that
\begin{itemize}
\item $\gamma_x$ has linking number $i$ with the binding $P$, and
\item the projection $P\times D^2 \to P$ sends $\gamma_x$ to a $j$-fold cover of a fiber of $P$.
\end{itemize}
We denote the corresponding orbit space by $S_{i,j}=\Sigma_{i,j}/S^1$.

\subsection{Spanning disks and Maslov indices}
\label{sec:spanning_disks}
We now assume in addition to the conditions listed at the beginning of Section~\ref{sec:indices} that
\begin{itemize}
\item $c_1(Q)=c[\omega]$.
\item $n\geq 3$, $k=1$ and $\pi_1(Q)=0$. This guarantees that the fibers of the prequantization bundle are contractible, so we can find disks bounding these fibers.
\end{itemize}
In Lemma~\ref{lemma:CZ_special_orbit} and Lemma~\ref{lemma:index_control_2} we will point out that a small part of the setup here also works out in general, and we will also point out some topological conclusions that can be drawn if $\pi_1(Q)\neq 0$.

Given these assumptions we construct a spanning disk for a periodic Reeb orbit using the binding and page model.
In our setup, the inverse of the monodromy is given by $\psi^{-1}=Fl^{f_iR_{P}}_1$, where $f_i$ is the function defined in Lemma~\ref{lemma:decomposition} for a left-handed twist.

The equivalence relation~\eqref{eq:identification} then motivates the identification map, which we will use to ``straighten the mapping torus''
\[
\begin{split}
\Psi: P\times I \times \R & \longrightarrow P\times I \times \R \\
(p,t,\phi) & \longmapsto (Fl^{f_i(t)R_{P}}_1(p),t,\phi+U_i(t)\, ).
\end{split}
\]
With $U_i'(t)=-f_i'(t)e^t$ we compute the differential of $\Psi$ as
$$
T\Psi=
\left(
\begin{array}{ccc}
TFl^{f_i(t)R_{P}}_1 & f_i'(t) R_{P} & 0 \\
0 & 1 & 0 \\
0&-f_i'(t)e^t & 1
\end{array}
\right)
.
$$

\subsubsection{Construction of an annulus bounding an orbit}
\label{sec:annulus_trivialization}
We construct an annulus that glues to a disk in the binding model, and bounds a specific Reeb orbit.
For this, we choose $p_0\in P$. 
Define $Inv(r)=1-r$, and define a map from an annulus into the pages of the open book.
\begin{eqnarray*}
\psi_A:~S^1 \times I & \longrightarrow & P\times I \times \R/\sim \quad\subset  W \times \R/\sim \\
(\phi,r) & \longmapsto & (Fl^{R_P}_{f_i\circ Inv(r) \frac{\phi}{2\pi} } (p_0),Inv(r);U_i\circ Inv(r) \frac{\phi}{2\pi} ).
\end{eqnarray*}
This map is induced by the gluing map $\psi_G$ defined in Equation~\eqref{eq:gluingmap_binding_to_pages}. It is well-defined as
$$
(\phi+2\pi,r)\mapsto  (Fl^{R_P}_{f_i\circ Inv(r) \frac{\phi}{2\pi}+ f_i\circ Inv(r) } (p_0),Inv(r),U_i\circ Inv(r)  \frac{\phi}{2\pi} +U_i \circ Inv(r))\sim
 (Fl^{R_P}_{f_i\circ Inv(r) \frac{\phi}{2\pi} } (p_0),Inv(r),U_i\circ Inv(r) \frac{\phi}{2\pi} ).
$$
Denote the image of $\psi_A$ by $C$.
Since the contact form on $P\times I \times \R /\sim$ is given by $\alpha=d\phi+e^t \lambda_P$, we use the splitting $\ker \alpha=\ker \lambda_P\oplus \Span( \partial_t,\partial_\phi-e^{-t}R_P ) $.
Define $V$ to be the symplectic vector space $(\xi_P,d\lambda_P)|_{p_0}$.
Use the following map to trivialize the contact structure along the annulus $C$,
\begin{eqnarray*}
S^1 \times I \times V \oplus ({\R^2},\omega_0) & \longrightarrow & \xi|_C \\
(\phi,r;v,w_1,w_2)  & \longmapsto & (\psi_A(\phi, r);
\left(
\begin{array}{ccc}
TFl^{R_P}_{f_i\circ Inv(r)\frac{\phi}{2\pi}}|_{\xi_P} &f_i'\circ Inv(r) \frac{\phi}{2\pi} R_{P} & -e^{-Inv(r)}R_P \\
0 & 1 & 0\\
0 & -f_i'\circ Inv(r) e^{Inv(r)} \frac{\phi}{2\pi} & 1
\end{array}
\right)
\left(
\begin{array}{c}
v \\
w_1\\
w_2
\end{array}
\right)
)
.
\end{eqnarray*}
We claim that this trivialization is well-defined.
Indeed, if we insert $\phi+2\pi$ instead of $\phi$, then we can recognize the image as a composition with $T\Psi$.

With respect to this trivialization, which we denote by $\epsilon=\epsilon_{\xi_P}\oplus \epsilon_w$, the path of symplectic matrices of the linearized flow is given by
$$
s\longmapsto 
\left(
\begin{array}{ccc}
TFl^{R_P}_{-f_i(t)s}|_{\xi_P} & 0 & 0 \\
0 & 1 & 0\\
0 & f_i'(t)e^t s & 1
\end{array}
\right)
,
$$
where we have written $t=Inv(r)$ (put $s=\frac{\phi}{2\pi}$).
Now apply the direct sum axiom and a variation of the normalization axiom for the Maslov index of symplectic paths, \cite{Robbin:Maslovindex};
note that the symplectic form for the $\epsilon_w$-part has a minus sign, so we obtain
$$
\mu(S_{i,j};\epsilon)= \mu( TFl^{R_P}_{-f_i(t)s}|_{\xi_P};\epsilon_{\xi_P} )+\mu(
\left(\,
\begin{array}{cc}
1 & 0\\
f_i'(t)e^t s & 1
\end{array}
\right)
\,
;\epsilon_w
)
=+2cj-\frac{1}{2} \sgn f_i'.
$$
Since this trivialization ``winds'' around the binding, we modify the trivialization by composing with a suitable loop of symplectic matrices.
This new trivialization extends over a disk, and the Maslov index gets an additional contribution of $2i$, where $i$ is the number of revolutions around the binding.
We conclude
$$
\mu(S_{i,j})=2i+2cj-\frac{1}{2} \sgn f_i'.
$$
Since the orbit space $S_{i,j}$ has dimension $2n-3$, we obtain a formula for the reduced index with \cite[Lemma~2.4]{Bourgeois:thesis}.
\begin{lemma}[Conley-Zehnder index after perturbation]
Let $f_{Morse}$ be a Morse function on the orbit space $S_{i,j}$. Lift this Morse function to an $S^1$-invariant function $\bar f_{Morse}$, and define the perturbed contact form $\alpha_\epsilon=(1+\epsilon\bar f_{Morse}) \alpha$.
Fix a constant $T_{threshold}$. Then for $\epsilon$ sufficiently small all periodic Reeb orbits of $\alpha_\epsilon$ with action less than $T_{threshold}$ correspond to critical points of $f_{Morse}$.
Furthermore, the reduced index of a periodic Reeb orbit corresponding to the critical point $a$ of $f_{Morse}$ is given by
\begin{equation}
\label{eq:reduced_index_orbitspaces}
\bar \mu(\gamma_a)=2i+2cj-3/2 -\frac{1}{2} \sgn f_i'+\ind_a f_{Morse}.
\end{equation}
\end{lemma}

\subsection{A special orbit}
The periodic orbits in $S_{i,0}$ are special in the sense that they always bound a spanning disk, even without any of the assumptions made in the beginning of Section~\ref{sec:spanning_disks}.
Namely, the ``flat disk'' $\{ p_0 \}\times D^2_{r_1}\subset P\times D^2$ provides a spannning disk, where $D^2_{r_1}$ is a disk of radius $r_1$ in an enlarged neighborhood of the binding. See Section~\ref{sec:fattening_binding}.

The methods from the previous section apply, and we find.
\begin{lemma}[Conley-Zehnder index of special orbits]
\label{lemma:CZ_special_orbit}
Let $f_{Morse}$ be a Morse function on the orbit space $S_{i,0}$. Lift this Morse function to an $S^1$-invariant function $\bar f_{Morse}$, and define the perturbed contact form $\alpha_\epsilon=(1+\epsilon\bar f_{Morse}) \alpha$.
Fix a constant $T_{threshold}$. Then for $\epsilon$ sufficiently small all periodic Reeb orbits of $\alpha_\epsilon$ with action less than $T_{threshold}$ correspond to critical points of $f_{Morse}$.
Furthermore, the reduced index of a periodic Reeb orbit corresponding to the critical point $a$ of $f_{Morse}$ is given by
\begin{equation}
\bar \mu(\gamma_a)=2i-1+\ind_a f_{Morse}.
\end{equation}
\end{lemma}

\subsection{Orbits through the content of the pages}
A monodromy given by a (left-handed) fibered Dehn twist $\tau^{-1}$ is the identity on the content of the page, so the manifold with boundary given by $W_{in}\times S^1 \subset \OB(W,\tau^{-1})$ consists of periodic Reeb orbits.
In this case we need to be somewhat careful with the contribution of the boundary.
We will choose a perturbation where the boundary of $W_{in}\times S^1$ contributes in a simple way.

Recall that we decomposed $W^{2n-2}=W_{in}\cup P \times [0,1]$.
Now choose a Morse function $f_{convex}$ on $W$ with the following properties.
\begin{enumerate}
\item $f_{convex}$ equals $\delta \cdot e^t$ on a symplectization piece $(P \times [0,1],d(e^t \lambda_P )\, )$ with coordinates $(p,t)$.
\item Periodic orbits of the Hamiltonian vector field $X_{f_{convex}}$ that do not correspond to critical points have large period, say much larger than $2\pi$.
\end{enumerate}
Some words on why this is possible.
For the first point we point out that we can realize any contact form on the boundary by attaching a piece of a symplectization.
For the second point, first choose a Morse function $\tilde f_{convex}$ which equals $e^t$ on the symplectization piece.
We can assume that the Hamiltonian vector field $X_{\tilde f_{convex}}$ has only finitely many periodic orbits in (a slightly shrunk copy of) $W_{in}$.
Denote the minimal period of a periodic orbit not corresponding to a critical point by $T_{min}$, and find $\delta>0$ such that $T_{min}/\delta> 2\pi$.
Then $f_{convex}=\delta \tilde f_{convex}$ has the required properties.

\begin{remark}
For later applications, it will be useful to have more control of the maximal index of $f_{convex}$.
If $W$ is a Weinstein manifold, then we can achieve this by choosing an $\Omega$-convex Morse function with the above properties.
\end{remark}

Define a Hamiltonian on $P\times I$ by
$$
F=f(t)\cdot e^t
$$
such that $f(t)+f'(t)=-(f_i(t)-2\pi)$.
Then the Hamiltonian vector field $X_F$ satisfies
$$
X_F=-(f(t)+f'(t)\,)\cdot R_P=(f_i(t)-2\pi)R_P,
$$
so its time $1$-flow generates a \emph{right-handed} fibered twist with profile $f_i$.
Note the signs and the right-left conventions are those from Remark~\ref{rem:conventions_twisting_profile}.
Furthermore, by choosing $F(0)=0$, we see that we choose $F$ to be $0$ on the content of the pages $W_{in}$.
Now define
$$
H=F+f_{convex},
$$
and let $X_H$ denote its Hamiltonian vector field of $H$.
Write the time $1$-flow of $X_H$ as
$$
\tau_{MB}=Fl_1^{X_H}.
$$
Observe that a right-handed fibered Dehn twist is symplectically isotopic, relative to the boundary, to $\tau_{MB}$ (put a parameter in front of $f_{convex}$ to make this isotopy).
Hence we consider the contact open book
$$
\OB(W,\tau_{MB}^{-1}).
$$
Since $\tau_{MB}$ is the flow of a Hamiltonian vector field,  $\tau_{MB}^*\lambda=\lambda-\mu$, where $\mu$ is exact.
In a collar neighborhood of the boundary we can use Formula~\eqref{eq:formula_u} to compute an explicit primitive of $\mu$.
\begin{remark}
The choice of the sign $+$ in front of $f_{convex}$ is a convenient choice for our purposes, as this will prevent the creation of additional orbit spaces.
\end{remark}

Later on, we shall only need the indices of periodic orbits that have linking number $1$ with the binding $P$. We have the following result.
\begin{lemma}
\label{lemma:index_control_left_twist}
Let $\gamma$ be a periodic Reeb orbit in $Map(W^{2n-2},\tau_{MB})\subset \OB(W,\tau_{MB}^{-1})$ such that the linking number of $\gamma$ with the binding equals $1$.
Then one of the following holds.
\begin{itemize}
\item $\gamma$ corresponds to a critical point of $f_{convex}$.
Its reduced index equals
$$
\bar \mu(\gamma)
=2n-2-\ind  f_{convex}-2c.
$$
\item $\gamma$ is a periodic orbit in the margin piece $P\times I \times \R /\sim$. In this case $\gamma$ comes in the Morse--Bott family $S_{1,0}$, and the reduced index of the periodic Reeb orbit corresponding to a minimum of a Morse function on $S_{1,0}$ equals
$\bar \mu(\gamma)=1$.
\end{itemize}
\end{lemma}

\begin{proof}
Let $x$ be a fixed point of the first iterate of $\tau_{MB}$.
Then $x\in W_{in}$ or $x\in P\times I$.

If $x\in W_{in}$, then we claim that $x$ is a critical point of $f_{convex}$.
Indeed, by property~(2) of $f_{convex}$, other periodic orbits of $X_{f_{convex}}$ have large period, so they cannot give rise to periodic Reeb orbits corresponding to a fixed point of the first iterate of the monodromy.
To determine the index of the periodic Reeb orbit $\gamma$ through $x$, we use a small modification of \cite[Lemma~2.4]{Bourgeois:thesis},
\[
\begin{split}
\bar \mu(\gamma) & =2i+2cj+n-3-\frac{1}{2}\dim W_{in}+(2n-2)-\ind_x H \\
 &=-2c+(2n-2)-\ind_x H
=2n-2-\ind_x  f_{convex}-2c,
\end{split}
\]
where we have put $i=1$. The twist in the $P$-direction equals $j=-1$ in the case of an orbit in the contents of pages.

On the margin piece $P\times I\times \R /\sim$, the Hamiltonian vector field $X_{-H}$, which generates the left-handed Dehn twist, equals $(2\pi-f_i(t)+\delta)R_P$.
So if an orbit $\gamma$ through $x=(p,t)\in P\times I$ has linking number $1$ with the binding, then $x$ lies in the perturbed copy of $S_{1,j}$ for some $j$.
As $2\pi-f_i(t)+\delta$ is an injective function, we only have to check where it takes values in $2\pi \Z$.
This is only the case if $f_i(t)=\delta$, so we see that $x$ lies in the perturbed copy of $S_{1,0}$.
Since $f_i'<0$, we obtain the reduced index of $\gamma$ by Formula~\eqref{eq:reduced_index_orbitspaces}.
\end{proof}

\begin{lemma}
\label{lemma:index_control_2}
Let $W$ be a Liouville filling for a prequantization bundle $(P,\lambda_P)$ over $(Q,k\omega)$ where $\omega$ is a primitive symplectic form and $k\in \Z_{>1}$.
Assume furthermore that the inclusion map $P\to W$ induces an injection on $\pi_1$.

Let $\gamma$ be a periodic Reeb orbit in $Map(W,\tau_{MB})$ such that the linking number of $\gamma$ with the binding equals $1$.
Then one of the following holds
\begin{itemize}
\item $\gamma$ is a periodic orbit in $S_{1,0}$.
Furthermore, the reduced index of the orbit corresponding to the minimum of a Morse function on $S_{1,0}$ is $\bar \mu(\gamma)=1$.
\item $\gamma$ lies in the content of the pages and corresponds to a critical point of $f_{convex}$.
Furthermore, if $\gamma_1$ is a periodic orbit in the orbit space $S_{1,0}$, then $[\gamma]\neq [\gamma_1]$ as a free homotopy class in $\OB(W,\tau_{MB}^{-1})-P\times \{ 0 \}$. 
\end{itemize}
\end{lemma}

\begin{proof}
The first assertion follows from Lemma~\ref{lemma:CZ_special_orbit}.
The first part of the second assertion follows from the proof of Lemma~\ref{lemma:index_control_left_twist}.
For the second part, note that after untwisting the mapping torus $W\times \R /\sim$, the curve $\gamma_1$ moves once along an $S^1$-fiber in the $P$-direction, whereas $\gamma$ does not.
If $k>1$, then this fiber is non-contractible, and this implies the claim.
\end{proof}

\section{Holomorphic curves near the binding of an open book}
\label{sec:holomorphic_plane}
In this section we will analyze holomorphic curves in the symplectization of the contact open book $Y:=\OB(W,\tau^{-1})$.
Here $W$ is a Liouville domain with prequantization boundary $P$, and $\tau^{-1}$ is a left-handed fractional twist.
Instead of taking the open book contact form, we start by using the $S^1$-invariant contact form constructed in Section~\ref{sec:invariant_contact}, and we will make additional perturbations in order to obtain non-degeneracy properties.
In a neighborhood of the binding the invariant contact form looks like
\begin{equation}
\label{eq:contact_form_fat_nbhd}
(P\times D^2,\alpha=h_1(r)\lambda_P +h_2(r) d\phi).
\end{equation}
\subsection{Fattening the binding: setup for finding finite energy holomorphic planes}
\label{sec:fattening_binding}
In order to have a single model containing both the binding and an important part of the pages, we ``fatten'' the binding using the following procedure.
In the previous section we perturbed the monodromy to $\tau_{MB}$.
We have $\tau_{MB}^*\lambda=\lambda-\mu$, where $\mu$ is exact.
In a neighborhood of the boundary of $W$, we follow Formula~\eqref{eq:formula_u} and obtain an explicit primitive $U_0$ with $dU_0=\mu$.
For $x=(p,t)\in P\times [0,1]$ we have
$$
U_0(x)=U_0(t)=
C-\int_{s=a}^t f_i'(s)e^sds,
$$
where $f_i$ is the twisting profile for a left-handed twist as defined in Lemma~\ref{lemma:decomposition}. 
This primitive $U_0$ extends to $W_{in}$.
Choose $C\in \R$ such that $U_0>0$ and
\begin{equation}
\label{eq:condition_C}
\frac{\max_{x\in W} U_0(x)}
{\min_{x\in W} U_0(x)}<2.
\end{equation}
This condition on $U_0$ is not necessary to prove the main result, but it is necessary for Section~\ref{sec:HC=0} which concerns an application to contact homology.
This condition is then used to control the action of periodic Reeb orbits.

Define $\alpha_{FB}:=\psi_G^* (d\phi+ e^t \lambda_P)$.
By Formula \eqref{eq:pullback_form} we find
\[
\begin{split}
\alpha_{FB} & =e^{1-r} \lambda_P+ \frac{ \tilde C+\int_0^{1-r} e^s f_i(s) ds}{2\pi}d\phi.
\end{split}
\]
We are going to extend the coefficient of $\lambda_P$ to a function $h_1$, and the coefficient of $d\phi$ to a function $h_2$.

\begin{lemma}
\label{lemma:orbits_near_binding}
There are functions $h_1:[0,r_1+\delta]\to \R$ and $h_2:[0,r_1+\delta]\to \R$ such that
\begin{itemize}
\item the form $\alpha=h_1(r)\lambda_P+h_2(r)d\phi$ is a smooth contact form on $P\times D^2$ extending $\alpha_{FB}$.
\item The function $h_2$  has a unique maximum in $r_1$, where $Inv(r_1):=1-r_1=t_1$. Here $t_1$ is the unique zero of the function $f_i=-f_m$, the twisting profile for a left-handed twist as defined in Section~\ref{sec:right_left_binding_twist}.
\item Periodic Reeb orbits $\gamma_{ij}(t)=(\gamma_P(j t);r,it)$ with $r<r_1$ either are binding orbits (orbits with $r=0$), or satisfy $\lk(P\times \{ 0 \},\gamma_{ij})=i\geq 2$.
\item Every binding orbit has a larger action than the action of $\gamma_1(t)=(p_0;r_1,t)$.
\end{itemize}
\end{lemma}

\begin{proof}
We extend $h_2$ such that it has a unique maximum in $r_1$. 
It is increasing on $[0,r_1]$, and near $0$, we require $h_2(r)=r^2$.
For $h_1$ we choose a decreasing function with the following properties.
\begin{itemize}
\item $h_1(r)=e^{1-r}$ near $r=r_1$.
\item for $r\in]0,r_1]$ the derivative is sufficiently negative such that $\vert \frac{h_2'}{h_1'}\vert<1$.
\item $h_1(0)>h_2(r_1)$.
\end{itemize}
These choices guarantee that the first two asserted properties hold.
To see that the third property holds, take $r$ with $0<r<r_1$.
We rescale the Reeb vector field to
$$
-\frac{h_2'(r)}{h_1'(r)}R_{P}+\partial_\phi.
$$
Since the coefficient of $R_{P}$ is non-zero, but less than $1$ in absolute value, a Reeb orbit cannot close up after a single revolution around the binding, so $\lk(P\times \{ 0 \},\gamma_{ij})=i\geq 2$.

For the last assertion, plug in the assumption $h_1(0)>h_2(r_1)$ into Formula~\eqref{eq:action_near_binding}.
\end{proof}
From now we will use the functions $h_1$ and $h_2$ provided by this lemma.
Our first goal is
\begin{lemma}
\label{lemma:existence_rigid_regular_plane}
Let $W$ be a Liouville domain whose contact type boundary $\partial W=(P,\lambda_P)$ is a prequantization bundle over a symplectic manifold, and suppose it admits a right-handed fractional twist $\tau$.
Consider the contact open book $Y=\OB(W,\tau^{-1})$, and fatten the binding model such that the function $h_2$ has a unique maximum in $r_1$, as in the above.
Then the following holds true.
\begin{enumerate}
\item There is a periodic Reeb orbit $\gamma_1(t)=(p_0;r_1,t)$ in the fattened binding $P\times D^2\subset Y$.
\item This orbit $\gamma_1$ bounds a unique, rigid finite energy holomorphic plane in the symplectization $\R \times Y$. Furthermore, we can assume that this finite energy plane is regular.
\end{enumerate}
\end{lemma}
\begin{proof}
The proof is somewhat long, so we split it into several sections.
Part (1) is clear by the construction in the previous section.

Existence of a finite energy holomorphic plane in the Morse--Bott setup is proved in Lemma~\ref{lemma:existence_finite_energy_plane}.
Uniqueness in the Morse--Bott setup is proved in Section~\ref{sec:uniquess_finite_energy_plane}.
In the sections after that we prove regularity of the finite energy plane and describe how to go from the Morse--Bott setup to the non-degenerate setup.
\end{proof}

\begin{remark}
In any open book, not just those with left-handed twists in the monodromy, one can always perturb $h_2$ to have such a maximum, so this lemma does not have direct geometric meaning.
Instead, in special setups one can show that $\gamma_1$ does not bound any other holomorphic curves.
If this happens, the finite energy plane from the lemma will have a meaning: it obstructs fillability, and guarantees the existence of a contractible periodic Reeb orbit, even after deformations of the contact form by Lemma~\ref{lemma:non-fillability}.

We also want to point out that the maximum of $h_2$ in the case of a left-handed twist arises rather naturally, whereas perturbing $h_2$ will in general create additional periodic Reeb orbits and holomorphic curves.
\end{remark}

\subsection{An adjusted almost complex structure for the symplectization}
The binding $P$ is a prequantization bundle over an integral symplectic manifold $(Q,\omega)$. We denote the projection from $P \to Q$ by $\pi_Q$.
Choose a compatible almost complex structure $J_Q$ for $(Q,\omega)$.
We lift this to a complex structure compatible with $(\xi_P=\ker \lambda_P,d\lambda_P)$ using the formula
$$
J_{\xi_P}:=-\pi_Q^*J_Q:=-Hor \circ J_Q \circ d\pi_Q.
$$
Here the map $Hor_p$ takes horizontal lifts of vectors in $T_{\pi(p)}Q$ to the contact structure $\xi|_p$ in $P$.
We need the minus sign because of our convention~\eqref{eq:d_connection}.

A fattened neighborhood of the binding as defined in Section~\ref{sec:fattening_binding}, has the form $P\times D^2$. Denote the natural projection $P\times D^2 \to P$ by $\pi_P$.
Define
$$
X_1=\partial_r \quad \text{and} \quad X_2=\frac{1}{\det H}\left( -h_2 R_{P}+h_1 \partial_\phi \right)
.
$$
Split the contact structure $\xi=\ker \alpha$ as
$$
\xi=\pi_P^*\xi_P \oplus \R X_1 \oplus \R X_2.
$$
The vectors $X_1,X_2$ form a symplectic basis of the complement of $\pi_P^*\xi_P$ in $\xi$, so we define a complex structure on this complement by putting $J_{X_1X_2}(X_1)=X_2$ and $J_{X_1X_2}(X_2)=-X_1$, and extending linearly.
Then
$$
J_\xi:=\pi_P^* J_{\xi_P}+J_{X_1X_2}
$$
defines a compatible complex structure on $(\xi,d\alpha)$. More explicitly, we can write
$$
J_{X_1X_2}=X_2\otimes X_1^*-X_1\otimes X_2^*=\frac{1}{\det H}\left( -h_2 R_{P}+h_1 \partial_\phi \right) \otimes dr-\partial_r \otimes \left( h_1' \lambda_P +h_2' d\phi \right).
$$

We extend the complex structure $J_\xi$ to an adjusted almost complex structure for the symplectization using the standard recipe
$$
J\partial_t =R_\alpha.
$$
We obtain
$$
J=R_\alpha \otimes dt - \partial_t \otimes \alpha+J_{X_1X_2}+\pi_P^* J_{\xi_P}.
$$
In matrix notation
$$
J=
\left(
\begin{array}{ccccc}
0 & -h_1 & 0 & 0 & -h_2 \\
\frac{h_2'}{\det H} & 0 & 0 & \frac{-h_2}{\det H} & 0 \\
0 & 0 & J_{\xi_P} & 0 & 0 \\
0 & -h_1'& 0 & 0 & -h_2' \\
\frac{-h_1'}{\det H} & 0 & 0 &\frac{h_1}{\det H} & 0
\end{array}
\right)
,
$$
where we have ordered the coordinates/basis vectors as $(t,R_{P},\xi_P,r,\phi)$.

\subsection{PDE to ODE: constructing a rigid holomorphic plane}
We prove the following lemma.
\begin{lemma}
\label{lemma:existence_finite_energy_plane}
There is a finite energy holomorphic plane $u_0$ in the symplectization of $P\times D^2$ asymptotic to $\gamma_1$.
\end{lemma}
\begin{proof}
We use the Morse-Bott almost complex structure from the previous section.
Take $p_0$ in $P$ and make the following ansatz for the holomorphic plane
\begin{eqnarray*}
u: \C & \longrightarrow & \R \times P \times D^2 \\
(\rho,\psi) & \longmapsto & (t(\rho);p_0,r(\rho),\psi)
\end{eqnarray*}
With respect to the coordinates/basis vectors $(t,R_{P},\xi_P,r,\phi)$ we have
$$
u_\rho:=\frac{\partial u}{\partial \rho}=
\left(
\begin{array}{c}
t_\rho \\
0 \\
0 \\
r_\rho \\
0
\end{array}
\right)
\text{ and }
u_\psi=
\left(
\begin{array}{c}
0 \\
0 \\
0 \\
0 \\
1
\end{array}
\right)
$$
The Cauchy-Riemann equation in polar coordinates, $u_\rho+\frac{1}{\rho}J u_\psi=0$, reduces hence to a system of ODE's
\[
\begin{split}
t_\rho(\rho) &=\frac{1}{\rho} h_2(r(\rho) \,) \\
r_\rho(\rho) &=\frac{1}{\rho} h_2'(r(\rho) \,)
\end{split}
\]
This system can be solved for by first integrating the second equation, and substituting the solution in the first: integration yields then a solution, and there are two integration constants, one for shifting $t$ (the symplectization symmetry), and one for rescaling $\rho$ (an automorphism of the plane).

We check that this is a smooth solution and that it is asymptotic to $\gamma_1$.
To see that the solution is smooth, we use the standard form of $h_2(r)=r^2$ near $r=0$. This is not really necessary, but it simplifies the argument.
Near $r=0$, we get
\[
\begin{split}
r_\rho(\rho) &=\frac{1}{\rho} 2 r(\rho)
\end{split},
\]
so $r(\rho)=C_r \rho^2$, where $C_r$ is an integration constant.
The first equation reduces to
\[
\begin{split}
t_\rho(\rho) &=\frac{1}{\rho} C_r^2 \rho^4.
\end{split}
\]
So near $r=0$, the solution looks like $t(\rho)=C_t+\frac{1}{4}C_r^2\rho^4$, where $C_t$ is another integration constant.
We see that the map $u$ is smooth near $r=0$. Away from $r=0$, there are no coordinate singularities, so $u$ is smooth everywhere.
Finally, we check that $u$ is asymptotic to $\gamma_1$.

To see this, note that as $\rho\to \infty$, the radial component $r(\rho)\to r_1$.
There $t_\rho(\rho)\sim \frac{h_2(r_1)}{\rho}$, so asymptotically $t\sim h_2(r_1) \log(\rho)$.
After going to cylindrical coordinates near $\infty$ (such coordinates are described after Lemma~\ref{lemma:normal_coordinates}) and noting that $2\pi h_2(r_1)=\mathcal A(\gamma_1)$ we see that $u$ converges exponentially to the Reeb orbit $\gamma_1(\psi)=(p_0;r_1,\psi)$ in $P\times D^2$. 
For later use, we put $u_0:=u$.
\end{proof}

\subsection{Uniqueness of holomorphic planes}
\label{sec:uniquess_finite_energy_plane}
Consider the projection
\begin{eqnarray*}
\pi: \R \times P \times D^2 & \longrightarrow & D^2.
\end{eqnarray*}
We shall start by arguing that any finite energy holomorphic plane $u$ asymptotic to $\gamma_1$ has the same projection $\pi(u_0)$ as the plane $u_0$ we found in Lemma \ref{lemma:existence_finite_energy_plane}.

\begin{lemma}
\label{lem:holomorphic_curve_projection}
Let $u:\C \to \R \times Y$ be a finite energy holomorphic plane asymptotic to $\gamma_1$.
Then $\pi \circ u$ is defined, and $\pi\circ u(\C)=\{ (r,\phi) \in D^2~|~r\leq r_1 \}$.
\end{lemma}
The idea is that curves that escape this neighborhood exceed the a priori energy bound. 
Here are more details.
\begin{proof}
Since the asymptote $\gamma_1$ has linking number $1$ with the binding, there is a unique point on $u$ intersecting the binding.
We will assume that this is $u(0)\in P\times \{ 0\}$.
Clearly, there is some connected, open subset $V\subset \C$ such that $\pi \circ u|_V$ is defined.
Furthermore, we can find $V$ such that $D^2_{r_1} \subset \pi\circ u(V)$ as $u$ can otherwise not be asymptotic to $\gamma_1$.
By Stokes' theorem the energy of the holomorphic curve $u$ is given by
$$
E(u)=\mathcal A(\gamma_1)=E(u_0)= 2\pi h_2(r_1)
.
$$
On the other hand, we can also compute the energy directly as
$$
E(u)=\int_\C u^* d\alpha \geq \int_{u(V)} h_1'dr \w \lambda_P+h_1d\lambda_P+h_2' dr \w \phi.
$$
As $u$ is $J$-holomorphic and $h_1 d\lambda_P(\cdot, J \cdot)=h_1 d\lambda_P(\cdot, J_{\xi_P} \cdot)$ is non-negative, this can be estimated by
\begin{equation}
\label{eq:estimate_energy_curve}
E(u)\geq \int_{u(V)} h_1'dr \w \lambda_P+h_2'dr \w d\phi.
\end{equation}
We are going to give a lower bound for this quantity by dividing the plane into annuli for which the image has inner radius $r_{s_j}$ and outer radius $r_{e_j}$. 
Then the energy contains the following term
$$
\int_{A_{r_{s_j},r_{e_j}}} h_1'dr \w \lambda_P+\int_{A_{r_{s_j},r_{e_j}}} h_2'dr \w d\phi.
$$
The second term is positive for $r<r_1$, but the first term can be negative.
However, we can parametrize the inverse image of a regular value of $r$ as a circle, and this allows us to say more.
We find the following contribution to the energy for an annulus with radii $r_{s_j}$ and $r_{e_j}$,
\begin{eqnarray*}
\int_{A_{r_{s_j},r_{e_j}}} h_1'dr \w \lambda_P & = &
\int_{r=r_{s_j}}^{r_{e_j}} \int_{\psi=0}^{2\pi}h_1'(r) \lambda_P(\partial_\psi u) d\psi dr \\
& = & \int_{r=r_{s_j}}^{r_{e_j}} \int_{\psi=0}^{2\pi}\frac{h_1'(r)}{h_1(r)} \left(
\alpha-h_2(r) d\phi
\right)
(\partial_\psi u) d\psi dr \\
& = & \int_{r=r_{s_j}}^{r_{e_j}} -\frac{h_1'(r)}{h_1(r)} \left(
2\pi h_2(r) - \int_{\psi=0}^{2\pi} \alpha(\partial_\psi u) d\psi
\right)
 dr. \\
\end{eqnarray*}
Here we used the fact the asymptotics of $u$ are the same as those of $u_0$, so $\gamma_1$ is covered once, and $\int_0^{2\pi} d\phi(\partial_\psi u) d\psi =2\pi$.
We see that this contribution to the energy is non-negative as long as
$$
\int_{\psi=0}^{2\pi} \alpha(\partial_\psi u) d\psi \leq 2\pi h_2(r).
$$
This is the action of the loop for which the $r$-component of $u$ equals $u^r=r$.

Let us now complete the argument.
Consider an increasing sequence of regular values $\{ r_i \}_i$, where $r_i$ is the outer radius of an annulus, that converges to $r_1$.
There are two cases to consider.
\begin{itemize}
\item{} There is a subsequence, denoted also by $\{ r_i \}_i$, such that
$$
\int_{\psi=0}^{2\pi} \alpha(\partial_\psi u(r_i,\psi)\, ) d\psi > 2\pi h_2(r_i).
$$
If this happens, then the above estimate does not work, but in this case we can directly compute the action by Stokes' theorem to find
$$
\lim_{i\to \infty} \int_{u,r<r_i} d\alpha \geq E(u_0).
$$

\item{} If there is no such subsequence, then we can apply the above estimate.
In particular, we find that
$$
\int_{A_{r_{s_j},r_{e_j}}} h_1'dr \w \lambda_P \geq 0,
$$
so the energy of the annulus $A_{r_{s_j},r_{e_j}}$ is at least
$$
E_{A_{r_{s_j},r_{e_j}}}(u)\geq 2\pi h_2(r_{e_j})- 2\pi h_2(r_{s_j}).
$$
\end{itemize}
Combine this with Equation~\ref{eq:estimate_energy_curve} to obtain an estimate for the total energy of the curve $u$.
We have
$$
E(u)\geq E(u_0)
$$
if $\pi \circ u$ just covers $D^2_{r_1}$.

Arguing by contradiction we can show the claim of the lemma.
If there is a point $x$ such that $\pi \circ u(x)\notin D^2_{r_1}$, then we find an open set $U_x$ not contained in $D^2_{r_1}$ which gives a positive contribution to the energy.
In particular,
$$
E(u)>E(u_0).
$$
But this is impossible by Stokes' theorem which asserts that we must have equality.
\end{proof}

We shall now reduce the problem to a holomorphic curve in a $3$-dimensional contact manifold.
Consider a holomorphic plane
\begin{eqnarray*}
u: \C & \longrightarrow & \R \times P\times D^2_{r_1} \\
 z & \longmapsto & (f(z),g(z),h(z)  )
\end{eqnarray*}
with the same asymptotics as $u_0$.
By Lemma~\ref{lem:holomorphic_curve_projection} this curve $u$ must have the same projection to $D^2_{r_1}$ as $u_0$, and $u$ stays in a set of the form $\R \times P \times D^2_{r_1}$.
Consider the projection
\[
\begin{split}
\bar \pi_Q:~\R \times P \times D^2_{r_1} & \longrightarrow Q \\
(t,p,z) & \longmapsto \pi_Q(p).
\end{split}
\]
\begin{lemma}
The map $\bar \pi_Q\circ u: \C \to Q$ is a $(-J_Q)$-holomorphic curve with vanishing area. Hence $\bar \pi_Q\circ u$ is constant.
\end{lemma}
\begin{proof}
This holds because $J$ is obtained by pulling back $J_Q$. More precisely, if we let $j$ denote the standard complex structure on $\C$, then we have
\[
\begin{split}
d(\bar \pi_Q\circ u) \circ j&=d \bar \pi_Q \circ d u \circ j=d \bar \pi_Q \circ J \circ du
=d\bar \pi_Q \circ \left( R_\alpha \otimes dt-\partial_t \otimes \alpha-\pi_P^* \circ \pi_Q^*J_Q +J_{X_1X_2}\right) \circ du\\
&=
-d \pi_Q \circ \left( Hor \circ J_Q \circ d\bar \pi_Q \right)du =-J_Q \circ d(\bar \pi_Q\circ u),
\end{split}
\]
so $\bar \pi_Q\circ u$ is a $(-J_Q)$-holomorphic curve. Furthermore, $\int_\C ( \bar \pi_Q\circ u)^*\omega_Q=\int_\C \frac{-1}{2\pi} u^* d \lambda_P$ by virtue of $\lambda_P$ being a connection form.
As $\int_\C u^*d \lambda_P=0$ by the asymptotic boundary conditions (the $P$-component of $u$ converges to $p_0$ at $\infty$), we conclude that $\bar \pi_Q\circ u$ is a constant map.
\end{proof}
From this lemma we conclude that $u$ must have the following form,
\begin{eqnarray*}
u: \C & \longrightarrow & \R \times P\times D^2_{r_1} \\
 z & \longmapsto & (f(z),Fl^{R_{P}}_{\tilde g(z)}(p_0),h(z)  )
\end{eqnarray*}
It follows that we can construct a holomorphic curve $v$ in the symplectization of a $3$-dimensional contact manifold,
\begin{eqnarray*}
v: \C & \longrightarrow & \R \times S^1 \times D^2_{r_1} \\
 z & \longmapsto & (f(z),\tilde g(z),h(z)  )
.
\end{eqnarray*}
Similarly, we can also associate a holomorphic curve $v_0$ with $u_0$.
\begin{eqnarray*}
v_0: \C & \longrightarrow & \R \times S^1 \times D^2 \\
 z & \longmapsto & (f_0(z),\theta_0,h_0(z)  )
.
\end{eqnarray*}
Let us now argue that $v$ is a translation of $v_0$.
This implies in turn that $u$ is a translation of $u_0$.
There are two cases to consider.
\begin{itemize}
\item{} The function $\tilde g$ is constant, so $\tilde g(z)=\theta_0$.
In this case, the curves $v$ and $v_0$ are both contained in the $3$-dimensional submanifold $\R \times \{ \theta_0 \} \times D^2$, and they both converge exponentially to the same periodic orbit $\gamma_1$.

If $v$ is not a translation of $v_0$, then we can apply a translation in the $\R$-direction of the symplectization to $v$, and obtain a shifted curve $v_{shift}$ that intersects $v_0$ non-trivially and transversely.
It follows that the intersections are $1$-dimensional submanifolds.
By positivity of intersection (cf.~\cite[Theorem~E.1.2]{MS:J_curves}), $v$ must coincide with $v_0$, which contradicts our assumption.

\item{} The function $\tilde g$ is not constant.
Then we define a shifted holomorphic curve with different asymptotics by
$$
v_{shift}= (f(z),\tilde g(z)+\theta_{shift},h(z)  ).
$$
The holomorphic curve $v_0$ is asymptotic to the periodic Reeb orbit $\gamma_1(\psi)=(\theta_0;r_1,\psi)$, and $v_{shift}$ is asymptotic to the periodic Reeb orbit $\gamma_{shift}(\psi)=(\theta_0+\theta_{shift};r_1,\psi)$.
The obvious Seifert surfaces for $\gamma_1$ and $\gamma_{shift}$ (flat disks in the solid torus $S^1\times D^2$) do not intersect, so $\gamma_1$ and $\gamma_{shift}$ do not link for any shift $\theta_{shift} \notin 2\pi \Z$.

On the other hand, the linking number $\lk(\gamma_1,\gamma_{shift})$ can also be computed as a $4$-dimensional intersection number of the Seifert surfaces of $\gamma$ and $\gamma_{shift}$.
We shall take the Seifert surfaces provided by the finite energy planes $v_0$ and $v_{shift}$, so
$$
\lk(\gamma_1,\gamma_{shift}) =v_0 \cdot v_{shift}.
$$
Since $\tilde g(z)$ is assumed to be non-constant, we can find $\theta_{shift}$ such that $v_0$ and $v_{shift}$, after possibly translating in the $\R$-direction, intersect. By positivity of intersection for holomorphic curves in dimension $4$, it follows that $\lk(\gamma_1,\gamma_{shift})>0$.
This is a contradiction, so we conclude that for finite energy planes asymptotic to $\gamma_1$ the function $\tilde g$ is constant.
\end{itemize}

\subsection{Regularity of the finite energy plane: Banach space for Morse-Bott setup and reduction to $3$ dimensions}
\label{seq:MB_regularity}
We will show that $u_0$ is regular, meaning that the linearized Cauchy-Riemann operator is surjective at $u_0$.
Since the details are rather lengthy, we give a summary first.
\begin{itemize}
\item We show that the linearized operator splits into a $4$-dimensional part, which is a mixture of normal and tangential directions, and higher-dimensional part, which is purely normal.
\item The kernel of the higher-dimensional part can be explicitly determined.
\item By automatic transversality results of \cite{Wendl:transversality} the $4$-dimensional part is regular. We conclude that $\ker D_{u_0}=\ind D_{u_0}$ for the full problem, so the cokernel is trivial.
\end{itemize}
Let us now give some details.
Since we are looking only at finite energy planes, we shall restrict ourselves to the functional analytic setup for this particular case.

Fix a contact manifold $(Y,\alpha)$, and let $\gamma_1$ be a periodic Reeb orbit of Morse-Bott type in $Y$.
Suppose that the action $\mathcal A(\gamma_1)=T$.
Fix a small number $\delta>0$. This number will serve as an asymptotic weight and is necessary for the linearized operator to be Fredholm. Also choose $p>2$.

We take the following lemma from \cite[Lemma~3.1]{Bourgeois:thesis}.
\begin{lemma}[Normal coordinates for a periodic orbit in a Morse-Bott manifold]
\label{lemma:normal_coordinates}
Suppose that $N_T$ is a $k$-dimensional submanifold in $(Y^{2n-1},\alpha)$ consisting of simple periodic Reeb orbits of period $T$.
Let $\gamma_1\subset N_T$ be a simple periodic orbit. 
Then there is a neighborhood $S^1\times D^{k-1}\times D^{2n-1-k}=S^1 \times D^{2n-2}$ near $\gamma_1$ with coordinates $(\phi;z_t,z_n)=(\phi;z)$ such that
$$
\alpha=g\cdot(d\phi+\frac{i}{2}(zd\bar z-\bar z dz)\, ),
$$
where $dg|_{N_T}=0$.
\end{lemma}
We apply this lemma to obtain coordinates $(t;\phi,z)$ near a cylinder over $\gamma_1$ for the symplectization $\R \times Y$. 
We then say that a map $u:\C \P^1-\{ pt \}=\C \to \R \times Y$ is {\bf asymptotically cylindrical} to $\gamma_1$ if there are $t_0,\phi_0$ such that in cylindrical coordinates $(\rho,\psi)\in Z_+=\R_{\geq 0}\times S^1$ for $\C$, the map $u$ satisfies
\[
\begin{split}
t\circ u-T \rho -t_0 & \in W^{1,p}_\delta(Z_+,\R) \\
\phi\circ u-\psi-\phi_0 & \in W^{1,p}_\delta(Z_+,\R) \\
z\circ u=(z_t,z_n)\circ u & \in W^{1,p}_\delta(Z_+,\R^{2n-2}).
\end{split}
\] 
Denote the orbit space $N_T/S^1$ by $S_T$.
Define the Banach manifold 
$$
\mathcal B^{1,p}_{MB,\delta}(\gamma_1):=\{ u:\C \to \R \times Y~|~u \text{ is of class } W^{1,p}_{loc}, \text{ and asymptotically cylindrical to }\gamma_1\}
$$
For a map $u\in \mathcal B^{1,p}_{MB,\delta}(\gamma_1)$ with $[\gamma_1]$ in the orbit space $S_T$, we define the finite-dimensional vector spaces
$$
V_{\gamma_1}=\R \frac{\partial}{\partial t}\oplus \R R_\alpha \text{ and }
W_{\gamma_1}=T_{[\gamma_1]} S_T.
$$
Then the tangent space at $u$ can be identified with
\[
\begin{split}
T_u \mathcal B^{1,p}_{MB,\delta}(\gamma_1) &\cong W^{1,p}_\delta(\C,u^* T(\R \times Y )\, ) \oplus V_{\gamma_1}\oplus W_{\gamma_1}.
\end{split}
\]
This vector space is actually enough for our purposes since we only need to check regularity near one solution, namely $u_0$.

Let $Y=\OB(W,\tau^{-1})$ be the contact manifold constructed in Theorem~\ref{thm:left_and_right_twisted} with the modified contact form from Section~\ref{sec:fattening_binding}.
Consider the holomorphic plane $u_0$ constructed in Lemma~\ref{lemma:existence_finite_energy_plane}.
We linearize the Cauchy-Riemann operator near the solution $u_0$.
Observe that the $P$-component of $u_0$ equals $p_0$.
On a neighborhood of $\R \times \{ p_0 \} \times D^2$ we choose the Riemannian metric
$$
g:=dt \otimes dt+g_{flat}^P+dr \otimes dr +r^2 d\phi \otimes d\phi.
$$
Let $\nabla$ denote the Levi-Civita connection for this flat metric.
Then $\nabla \partial_t=0,\nabla \partial_r=0, \nabla \partial_\phi=0$.
The linearized operator at a solution $u_0$ of the Cauchy-Riemann equations, acting on a vector field $X$, is then given by $D_{u_0}X \nabla X + J \nabla X \circ i +(\nabla_X J)\circ du_0 \circ i$, see \cite[Equation~3.8]{Wendl:transversality}. 
By plugging in $\partial_\rho$ we get a PDE with asymptotic boundary conditions given by the functional analytic setup,
$$
D_{u_0}(X)(\partial_\rho)= \nabla_{\partial_\rho} X + \frac{1}{\rho} J \nabla_{\partial_\psi} X +\frac{1}{\rho} (\nabla_X J)\partial_\phi.
$$
The last term can be simplified: $(\nabla_X J)\partial_\phi= \nabla_X (J\partial_\phi)-J\nabla_X \partial_\phi=-\nabla_X \left( h_2 \partial_t +h_2' \partial_r \right)$. 
Simplifying notation as well, we arrive at
$$
D_{u_0}(X)(\partial_\rho)= \nabla_\rho X + \frac{1}{\rho} J \nabla_\psi X - \frac{1}{\rho} X^r h_2' \partial_t-\frac{1}{\rho} X^r h_2'' \partial_r. 
$$
In order to understand the associated differential operator, consider the vector fields along the solution $u_0$ given by
\[
\begin{split}
X_1^I=\partial_t,\quad X_2^I=R_\alpha, \quad X_3^I=X_1, \quad X_4^I=X_2, \\
\{ X_i^{II} \} \text{ symplectic basis of }{\xi_P}|_{p_0}.
\end{split}
\]
Denote the span of the $X^I$-vectors by $E_{I}$, and the span of the $X^{II}$-vectors by $E_{II}$.
We obtain a splitting $u_0^*T(\R \times Y)=E_{I}\oplus E_{II}$.
Furthermore, with respect to this decomposition, the ``Morse-Bott'' vector space $W_{\gamma_1}$ decomposes into $W_{\gamma_1}^I\oplus W_{\gamma_1}^{II}$.
This gives rise to bounded linear projection maps
\[
\begin{split}
\pi_{I}:W^{1,p}_\delta(u_0^*T(\R \times Y)\,)\oplus V_{\gamma_1} \oplus W_{\gamma_1}
& \longrightarrow W^{1,p}_\delta(E_{I})\oplus V_{\gamma_1} \oplus W^I_{\gamma_1} \\
\pi_{II}:W^{1,p}_\delta(u_0^*T(\R \times Y)\,)\oplus V_{\gamma_1} \oplus W_{\gamma_1}
& \longrightarrow W^{1,p}_\delta(E_{II})\oplus W^{II}_{\gamma_1}
\end{split}
\]
splitting the tangent space $T_{u_0} \mathcal B^{1,p}_{MB,\delta}$.
Similarly, there is a splitting of the target space of the linearized operator as well.
We have
$$
L^p_\delta(\overline{\Hom}(T\C, u_0^*T(\R \times Y)\,)\, )
=
L^p_\delta(\overline{\Hom}(T\C,E_{I}\, ))
\oplus
L^p_\delta(\overline{\Hom}(T\C,E_{II}\, ))
,
$$
where $\overline{\Hom}$ denotes complex anti-linear maps.
\begin{lemma}
The vertical differential at $u_0$, 
\[
D_{u_0}: T_{u_0} \mathcal B^{1,p}_{MB,\delta} \longrightarrow
L^p_\delta(\overline{\Hom}(T\C, u_0^*T(\R \times Y)\,)\, ),
\]
splits as $D_{u_0}=D_{I}+D_{II}$, where 
\[
\begin{split}
D_{I}: W^{1,p}_\delta(E_{I})\oplus V_{\gamma_1} \oplus W^I_{\gamma_1} & \longrightarrow L^p_\delta(\overline{\Hom}(T\C,E_{I}\, )) \\
D_{II}: W^{1,p}_\delta(E_{II})\oplus W^{II}_{\gamma_1}
& \longrightarrow L^p_\delta(\overline{\Hom}(T\C,E_{II}\, ))
\end{split}
\]
\end{lemma}

\begin{proof}
For this, we write out the vertical part of the linearized operator.
Note that the pulled back tangent space $u_0^*T(\R \times P \times D^2)$ is trivialized by the ${X^{I}}$-vectors $\partial_t,R_\alpha,X_1=\partial_r,X_2=J_{X_1X_2} \partial_r$ and the $X^{II}$-vectors which form a symplectic basis of $\xi_P$ at $p_0$.
A dual basis is given by $dt,\alpha, dr, \beta,\{ X^{II,*}_j \}_j$, where
$$
\beta=h_1' \lambda_P+h_2' d\phi.
$$
The linearized Cauchy-Riemann equation $D_{u_0} \bar \partial_{J}(X)(\partial_\rho)=0$ can be written as the system
\begin{alignat*}{4}
\quad\quad
\partial_\rho X^t-\frac{1}{\rho}\alpha( \nabla_{\psi}X )-\frac{1}{\rho}X^rh_2' &= 0 
&\quad\quad\quad\quad(dt) \\
\quad\quad
\alpha(\nabla_\rho X)+\frac{1}{\rho}\partial_\psi X^t &=0
&\quad\quad\quad\quad(\alpha) \\
\quad\quad
\nabla_\rho X^{\xi_P}+\frac{1}{\rho}J_\xi \nabla_\psi  X^{\xi_P} &=0 
&\quad\quad\quad\quad(II\text{-part}) \\
\quad\quad
\partial_\rho X^r-\frac{1}{\rho}\beta(\nabla_\psi X)-\frac{1}{\rho}X^r h_2'' &=0 
&\quad\quad\quad\quad(dr) \\
\quad\quad
\beta(\nabla_\rho X)+\frac{1}{\rho}\partial_\psi X^r &=0.
&\quad\quad\quad\quad(\beta)
\end{alignat*}
The third equation $\nabla_\rho X^{\xi_P}+\frac{1}{\rho}J_\xi \nabla_\psi  X^{\xi_P} =0$ separates from the others, and this equation gives $D_{II}$. 
\end{proof}

To check that the linearized operator is surjective, we show that $\dim \ker D_{u_0}=\ind D_{u_0}$.
We emphasize that we are in a Morse-Bott setup, so the index can be computed with Formula~\eqref{eq:ind_moduli} if replace the Conley-Zehnder index by a perturbed variant as defined in \cite[Section~3.2]{Wendl:transversality}. Using Lemma~\ref{lemma:CZ_special_orbit} we may write the index as
$$
\ind D_{u_0}=5+\dim S_T.
$$
Now observe that if $X=X^I+X^{II}$ lies in the kernel of $D_{u_0}=D_{I}+D_{II}$ with $X^I\in W^{1,p}_\delta(E_{I})\oplus V_{\gamma_1} \oplus W^I_{\gamma_1}$ and $X^{II}\in W^{1,p}_\delta(E_{II})\oplus W^{II}_{\gamma_1}$, then $D_{II} X^{II}=0$.
Solutions to the equation $D_{II} X^{II}=0$ correspond to some of the Morse-Bott symmetries.
Indeed,  $D_{II} X^{II}=0$ is a standard system of Cauchy-Riemann equations on $\C$ in polar coordinates, so its solutions are holomorphic functions $\C \to \C^{\frac{\dim S_T-1}{2}}$.
In the functional analytic setup we have here, only constant solutions are admissible; other solutions do not lie in $ W^{1,p}_\delta(E_{II})\oplus W^{II}_{\gamma_1}$.
There are $\dim S_T-1$ real linearly independent constant solutions, so $\dim \ker D_{II}=\dim S_T-1$.

Since $D_{I}$ and $D_{II}$ decouple, we find also $D_{I} X^{I}=0$.
Now notice that $D_{I}$ corresponds to the linearized Cauchy-Riemann operator for the $3$-dimensional problem: the case that $W$ is a surface, and $P$ a collection of circles.
Denote the finite energy plane for the $3$-dimensional case by $u_{0,\dim=3}$.

For this $3$-dimensional case we use Wendl's automatic transversality result, \cite[Theorem 1]{Wendl:transversality}.
Since the curve $u_{0,\dim=3}$ is embedded, has genus $0$, only one positive puncture, the conditions of Wendl's theorem hold, and we see that $\dim \ker D_I=\ind D_{ u_{0,\dim=3} }=5+1$.
We conclude that for the original problem $\dim \ker D_{u_0}=\dim \ker D_{I} +\dim \ker D_{II}=6+\dim S_T-1=\ind D_{u_0}$, so the cokernel of $D_{u_0}$ is trivial.

\begin{remark}
Alternatively, we can use a more direct argument to compute $\ker D_{I}$.
In \cite{BvK} Fourier analysis was used to directly compute the kernel.
Also note that the index that we use here is not the same as the one in \eqref{eq:ind_moduli}; here we are working with maps, so there are still automorphisms.
\end{remark}

\subsection{From Morse-Bott to non-degenerate}
\label{sec:MB_non_deg}
We follow the procedure from Bourgeois' thesis to get a non-degenerate contact form.

Instead of performing the procedure in general, we will perform this procedure in our particular setup.
In our case, the Morse-Bott manifold is given by $P\times \{ r_1 \} \times S^1\subset P \times D^2$, and the orbit space is $P=P\times S^1 /Reeb$.
Choose a Morse function $f$ on the orbit space with a unique local minimum in $[\gamma_1]=p_0$ and with value $f(p_0)=0$.
Lift this function to an $S^1$-invariant function $\bar f$ on the Morse-Bott manifold $P\times S^1$.

Define the perturbed contact form
$$
\alpha_\epsilon:=(1+\epsilon \bar f) \alpha.
$$
Denote the Reeb field of $\alpha_\epsilon$ by $R_\epsilon$.
Since $p_0$ is a minimum, $\gamma_1$ is also a periodic orbit of $R_\epsilon$, and as $f(p_0)=0$, it has in fact the same parametrization.
In addition, for $\epsilon>0$ sufficiently small, $\gamma_1$ is a non-degenerate periodic orbit of $R_\epsilon$.

Next we define an adjusted almost complex structure $J_\epsilon$ for the symplectization by putting $J_\epsilon|_{\xi}=J|_{\xi}$, and extending by the usual recipe $J_\epsilon \partial_t =R_\epsilon$.
First apply Lemma~\ref{lemma:normal_coordinates} to obtain coordinates $(\phi;x,y)$ for a neighborhood of $\gamma_1$ of the form $S^1\times U_{P\times I}$ such that
$$
\alpha=g\cdot(d\phi+xdy-ydx).
$$
In these coordinates, we have identified $S^1\cong \R/\Z$. 
For later use, we define the period $T=\int_{S^1}\alpha=g(0;0,0)$.
Take a moving frame $\partial_t,\partial_\phi,\{ U_i,V_j \}_{i,j}$ on this neighborhood of $\gamma_1$, where $\{ U_i,V_j \}$ forms a unitary trivialization of $(\xi,d\alpha,J_\xi)$, meaning in particular that $J_\xi$ is the standard complex structure with respect to this trivialization.
Write $F=(1+\epsilon \bar f)g$, so $\alpha_\epsilon=F\cdot (d\phi+xdy-ydx)$.
With respect to the above moving frame, $J_\epsilon$ is the matrix 
\begin{equation}
\label{eq:perturbed_J}
J_\epsilon=
\left(
\begin{array}{cccc}
0 & -F & 0 & 0 \\
\frac{1}{F} & 0 & 0 & 0 \\
\frac{V(F)}{F^2} & \frac{- U(F)}{F} & 0 & -\mathbbm{1} \\
\frac{-U(F)}{F^2} & \frac{- V(F)}{F} & \mathbbm{1} & 0
\end{array}
\right)
.
\end{equation}
Note that the first column is the Reeb vector field for $\alpha_\epsilon$, and that the restriction of $J_\epsilon$ to $\xi$ is equal to the restriction of unperturbed $J_0$ to $\xi$.

\subsubsection{Asymptotic behavior and Sobolev spaces}
In Section~\ref{seq:MB_regularity} we checked that the operator $D_{u_0}$ is surjective.
As $D_{u_0}$ is a Fredholm operator, we get a bounded right inverse $Q_{u_0}$.

Note that the difference of the endomorphisms
$$
\Delta_\epsilon:=J_\epsilon-J_0
$$
is small by the above formula~\eqref{eq:perturbed_J} for $J_\epsilon$.
Writing out the Cauchy-Riemann equation shows that there is a constant $C(p)$ depending on $p$ such that
$$
\Vert \bar \partial_{J_\epsilon} \Vert_{L^p_\delta} \leq C(p) \epsilon
$$
provided $\epsilon$ is sufficiently small.

We now investigate the behavior of solutions of the Cauchy-Riemann equation for the perturbed problem.
We need this to ensure that we can use the Sobolev space $W^{1,p}_\delta(\C,u^* T(\R \times P \times D^2 )\, ) \oplus V_{\gamma_1}$ also for the functional analytic setup in the non-degenerate case.
Near a vertical cylinder consider the coordinates
$$
Z=(Z^t,Z^\phi;Z_{P\times I})=(u^t-T\rho,u^\phi-\psi,u_{P\times I})\in \R \times S^1\times P\times I.
$$
With respect to these coordinates the Cauchy-Riemann equations for the cylindrical ends become
\[
0=u_\rho +J_{\epsilon} u_\psi=Z_\rho+T\partial_t+J_\epsilon Z_\psi +J_{\epsilon}\partial_\phi.
\]
We note that the second column of $J_\epsilon$ contains the gradient of $F$ in the direction of the contact structure with respect to the metric $\omega(\cdot,J_\xi\cdot )$, so the Cauchy-Riemann equations near the cylindrical ends reduce to
\[
Z_\rho+J_\epsilon Z_\psi+T(1-F/T) \partial_t
-\frac{1}{F} \grad f=0.
\]
Express the coordinates for $Z_{P\times I}$ into a part tangential to the Morse-Bott manifold $S_T=P$, denoted by $z_t$, and a normal part corresponding to the $I$-factor, denoted by $z_n$.
Rewrite the above equation into the general form
\begin{equation}
\label{eq:CR_near_vertical_cylinder}
Z_\rho+J_\epsilon Z_\psi+s_\epsilon(z_t,z_n) z_n
-S_\epsilon \grad f=0.
\end{equation}
With this equation and the proof of \cite[Proposition~A.2]{Bourgeois:MB_symplectic_homology} we obtain
\begin{proposition}
Let $N_T$ be the Morse-Bott submanifold consisting of points on simple periodic Reeb orbits with period $T$, and write $S_T=N_T/S^1$ for its orbit space.
Choose a Morse function $f:S_T \to \R$ on the orbit space $S_T$ with a single local minimum at $[\gamma_1]$.
Then there are $\delta,\epsilon_0>0$ such that, for $\epsilon<\epsilon_0$, every $J_\epsilon$-holomorphic plane $u:\C \to \R \times Y$ asymptotic to $\gamma_1$ satisfies
\begin{eqnarray*}
t \circ u(\rho,\psi)-T \rho-t_0 & \in &
W^{1,p}_\delta(\C,\R) \\
\phi \circ u(\rho,\psi)- \psi-\phi_0 & \in &
W^{1,p}_\delta(\C,\R) \\
z_{t} \circ u(\rho,\psi)- Fl^{S_\epsilon \grad f}_\rho(p_0) & \in &
W^{1,p}_\delta(\C,\R^{2n-3}) \\
z_{n} \circ u(\rho,\psi) & \in &
W^{1,p}_\delta(\C,\R^{1}) .
\end{eqnarray*}
for some $t_0\in \R$, $\phi_0 \in S^1$ and $p_0\in \R^{2n-3}$.
\end{proposition}

\subsubsection{Bounded right inverse for the non-degenerate setup and applying the implicit function theorem to construct a solution to the perturbed problem}
If $\epsilon$ is sufficiently small, then $D_{u_\epsilon}$ is still surjective.
By possibly choosing $\epsilon$ even smaller we can ensure that
$$
\Vert (D_{u_\epsilon}-D_{u_0}) Q_{u_0} \Vert <\frac{1}{2}
$$
holds.
Then we can define a bounded right inverse $Q_{u_\epsilon}$ for $D_{u_\epsilon}$ by putting
\[
\begin{split}
Q_{u_\epsilon}&:=Q_{u_0}(D_{u_\epsilon} Q_{u_0})^{-1}\\
&=Q_{u_0}(  (D_{u_\epsilon}-D_{u_0})Q_{u_0} +\id )^{-1}=Q_{u_0} \sum_{k=0}^\infty (-1)^k \left(  (D_{u_\epsilon}-D_{u_0})Q_{u_0} \right)^k.
\end{split}
\]
We now apply \cite[Proposition A.3.4]{MS:J_curves} and a modification of \cite[Theorem 3.5.2]{MS:J_curves} to our problem.
This will yield a solution $u_\epsilon:\C \to \R \times Y$ to the perturbed equation
$$
\begin{cases}
\bar \partial_{J_\epsilon} u_\epsilon=0,& \\
u_\epsilon \text{ asymptotic to } \gamma_1.
\end{cases}
$$
Moreover, this argument also shows that $u_\epsilon$ is a rigid curve, so a solution of the Cauchy-Riemann equation is unique (up to translation) for curves in a neighborhood of $u_0$ as in \cite[Corollary~3.5.6]{MS:J_curves}.

\subsubsection{Excluding other solutions to the perturbed problem}
We conclude by arguing that there are no finite energy holomorphic planes asymptotic to $\gamma_1$ other than translations of $u_\epsilon$ if $\epsilon$ is sufficiently small.

To show this, we argue by contradiction. Take a sequence $\{ \epsilon_n \}_n$ converging to $0$, and assume that there is a sequence of $J_{\epsilon_n}$-holomorphic curves $u_{n}$ that are not vertical translations of the $u_{\epsilon_n}$ we constructed before.
Since $u_0$ is regular and rigid, these curves $u_n$ cannot lie in a sufficiently small $C^0$-neighborhood $U_{C^0}$ of any unperturbed solution by a modification of \cite[Corollary~3.5.6]{MS:J_curves}.

Note that there is a subsequence of the $u_{n}$ converging to a $J_0$-holomorphic curve on compact subsets of $\C$.
The energy bound for $u_n$, namely
$$
E(u_n)\leq \mathcal A_{\alpha_\epsilon}(\gamma_1)=\mathcal A(\gamma_1),
$$
implies that we obtain a $J_0$-holomorphic finite energy plane $u_\infty$ asymptotic to some periodic Reeb orbit $\gamma$.
By energy/action considerations, this Reeb orbit $\gamma$ must give a point $[\gamma]$ in the Morse-Bott orbit space $S_T$.

We claim that $u_\infty$ is asymptotic to $\gamma_1$.
To see why, note that the positive cylindrical end of a holomorphic curve follows the positive gradient flow of $f$ by Equation~\eqref{eq:CR_near_vertical_cylinder}.
Any solution not converging to $\gamma_1$ is therefore pushed away from $\gamma_1$.
Since the total building, of which $u_\infty$ is just one layer, converges to $\gamma_1$, we obtain a contradiction if $u_\infty$ is not asymptotic to $\gamma_1$.

Hence $u_\infty$ is a solution to the unperturbed Cauchy-Riemann equations with asymptote $\gamma_1$, and by our previous uniqueness argument it is a translation of $u_0$.
It follows that the $J_{\epsilon_n}$-holomorphic curves $u_n$ must lie in some $C^0$ neighborhood $U_{C^0}$ of a solution to the unperturbed problem for sufficiently large $n$.
This is a contradiction.

\section{Other holomorphic curves: the situation away from the binding}
\label{sec:other_curves}
In this section we consider more general rational holomorphic curves asymptotic to $\gamma_1$.
As in the previous section~\ref{sec:holomorphic_plane}, $Y^{2n-1}$ denotes the contact manifold $Y=\OB(W,\tau^{-1})$ and we will use the contact form from the perturbation described in that section to make orbits non-degenerate.
We will show that the conditions of Lemma~\ref{lemma:non-fillability} hold given the assumptions in Theorem~\ref{thm:result}, to do so we use action, index and linking/homotopy arguments combined with some regularity arguments using Dragnev's Theorem~ \ref{thm:regularity_somewhere_injective}.
We start with a well-known observation about the linking number.
\begin{lemma}
\label{lemma:linking}
Let $\gamma^+,\gamma^-_1,\ldots,\gamma^-_m$ be periodic Reeb orbits that do not lie in the binding $P$ of a contact open book $Y$, and consider a holomorphic curve $u$ of genus $0$ with $[u]\in \mathcal M^A_0(\gamma^+;\gamma^-_1,\ldots,\gamma^-_m)$.
Then $\lk(P\times \{ 0 \},\gamma^+)\geq \sum_i \lk(P\times \{ 0 \},\gamma^-_i)$.
\end{lemma}
\begin{proof}
To see this, observe that the symplectization of $P$ is an almost complex submanifold of real codimension $2$ in $\R \times Y$.
Furthermore, the holomorphic curve $u$ can be thought of as a Seifert surface for the collection of oriented Reeb orbits $\Gamma$ that $u$ bounds, so $(\R \times P)\cdot u=\lk(P\times \{ 0 \},\Gamma)$.
By positivity of intersection, $\lk(P\times \{ 0 \},\Gamma)\geq 0$, so the claim follows.
\end{proof}

\begin{lemma}
\label{lemma:no_other_curves_fractional_fibered}
Let $W$ be a Liouville domain admitting a right-handed fractional twist $\tau$ of power $\ell>1$.
Let $Y=\OB(W,\tau^{-1})$.
The only rigid holomorphic curve in the symplectization of $Y$ with positive puncture $\gamma_1$ is the finite energy plane $u_0$ from Lemma~\ref{lemma:existence_finite_energy_plane}.
\end{lemma}
\begin{proof}
By Lemma~\ref{lemma:existence_rigid_regular_plane} there is a unique rigid, finite energy holomorphic plane asymptotic to $\gamma_1$.
Therefore we only need to exclude holomorphic curves that also have negative punctures.
Suppose that $u$ is such a holomorphic curve.
By Lemma~\ref{lemma:orbits_near_binding} the action of a binding orbit is larger than the action of $\gamma_1$, so none of the negative punctures can be a binding orbit.

Denote the binding of the contact open book by $P$. Then $\lk(P\times \{ 0 \},\gamma_1)=1$, so by Lemma~\ref{lemma:linking}, the linking number of any negative puncture can be at most $1$.
Since all page orbits intersect the pages in a positively transverse fashion, linking number $0$ is not possible.
So we conclude that there can be only one negative puncture, at which $u$ is asymptotic to $\gamma$.

The monodromy is nontrivial and fixed point free on the content of the page, since it corresponds to a non-trivial deck-transformation by the assumption $\ell>1$.
Therefore any periodic orbit in the content of the page must have linking number at least $2$. 
It follows that all candidates for $\gamma$ lie in the fattened binding and in the margins of the pages.
By Lemma~\ref{lemma:orbits_near_binding} we have $\lk(P\times \{ 0 \},\gamma)\geq 2$ if $r<r_1$.

The periodic orbit $\gamma$ must make an integer number of turns in $P$-direction, and the linking condition $\lk(P\times \{ 0 \},\gamma)=1$ tells us then that $f_i\circ Inv(r)$ must be an integer multiple of $2\pi$.
This only happens at $r_1$, so $\gamma$ satisfies $r=r_1$.
Since $\gamma_1$ corresponds to the minimum of a Morse function as described in Section~\ref{sec:MB_non_deg}, it has minimal action among all Reeb orbits corresponding to that orbit space.
Since $\mathcal A(\gamma)>\mathcal A(\gamma_1)$ is impossible for a holomorphic cylinder we have $\gamma=\gamma_1$. 
It follows that $u$ is a vertical cylinder, which is not a rigid curve.

We conclude that the claim holds.
\end{proof}

The next case combines a linking and index argument.
\begin{lemma}
\label{lemma:fibered_twist_neg_c_single_plane}
Let $W^{2n-2}$ be a Liouville domain with prequantization boundary $(P,\lambda_P)$ over $(Q,\omega)$, where $\omega$ is a primitive integral symplectic form, and $\pi_1(Q)=0$.
Assume furthermore that 
\begin{itemize}
\item $c_1(W)=0$.
\item $c_1(Q)=c[\omega]$.
\end{itemize}
Let $\tau$ denote a right-handed fibered Dehn twist on $W$.
Define $Y=\OB(W,\tau^{-1})$.
Denote the maximal index of a Morse function on $W$ that is convex near the boundary by $\max \ind$.
If $c\leq n-\frac{\max \ind+3 }{2}$, then the only rigid holomorphic curve in the symplectization of $Y$ with positive puncture $\gamma_1$ is the finite energy plane constructed in the proof of Lemma~\ref{lemma:existence_finite_energy_plane}. 
\end{lemma}

\begin{proof}
Let $u:\Sigma \to \R \times Y$ be any holomorphic curve with a single positive puncture, at which $u$ is asymptotic to $\gamma_1$.
As in the proof of the previous lemma, we use a linking argument to show that $u$ can only have one negative puncture, at which $u$ is asymptotic to $\gamma$.
Lemma~\ref{lemma:orbits_near_binding} excludes any orbit $\gamma$ that lies in a fattened neighborhood of the binding with $r<r_1$.
We invoke Lemma~\ref{lemma:index_control_left_twist} to conclude that $\gamma$ lies either in the content of the pages or $\gamma$ lies in $S_{1,0}$.
The case that $[\gamma]\in S_{1,0}$ was already excluded in the proof of the previous lemma.

We conclude that $\gamma$ lies in the content of the pages, so it corresponds to a critical point $a$ of $f_{convex}$ by Lemma~\ref{lemma:index_control_left_twist}.
Since $\gamma_1$ is simple, the holomorphic curve $u$ must be somewhere injective.
By Dragnev's theorem, we can perturb $J$ to make all somewhere injective curves regular.
Hence the moduli space $\mathcal M_0(\gamma_1;\gamma)$ is a smooth orbifold of dimension determined by the Fredholm index, which we compute now.

Denote the contact structure on $Y$ by $\xi$. 
We first show that $c_1(\xi)=0$.
For this we apply the Mayer-Vietoris sequence in cohomology,
\[
\begin{split}
0 \longrightarrow H^2(Y;\R) & \longrightarrow H^2(P\times D^2;\R) \oplus H^2(W\times S^1;\R) \\
c_1(\xi) & \longmapsto (i_{P\times D^2}^* c_1(\xi),i_{W\times S^1}^* c_1(\xi) \, )=(0,0).
\end{split}
\]
It follows that the Conley-Zehnder/Maslov indices do not depend on the homology class of a curve as explained in Remark~\ref{rem:CZ_via_disk}, so by Lemma~\ref{lemma:index_control_left_twist} and Formula~\eqref{eq:dimension_formula} we find the following dimension for the moduli space.
\[
\begin{split}
\dim \mathcal M_0(\gamma_1;\gamma)&=\bar \mu(\gamma_1)-\bar \mu(\gamma)=
1-(-2c+2n-2-\ind_a f_{convex} )\\
&=1+2c-2n+2+\ind_a f_{convex}.
\end{split}
\]
If $c$ satisfies the above assumptions, then this dimension is non-positive and regularity directly implies that this moduli space is empty.
\end{proof}

For the following lemma, a purely homotopical argument suffices to exclude other curves.
\begin{lemma}
\label{lemma:only_plane_with_pi1_cond}
Let $W^{2n-2}$ be a Liouville domain with prequantization boundary $(P,\lambda_P)$ over $(Q,k\omega)$, where $\omega$ is a primitive integral symplectic form, and $k\in \Z_{>1}$.
Let $\tau$ denote a right-handed fibered Dehn twist on $W$, and define $Y=\OB(W,\tau^{-1})$.
If the inclusion map $i:P=\partial W\to W$ induces an injection on $\pi_1$, then the only rigid holomorphic curve in the symplectization of $Y$ with positive puncture $\gamma_1$ is the finite energy plane constructed in the proof of Lemma~\ref{lemma:existence_finite_energy_plane}. 
\end{lemma}

\begin{proof}
The same arguments as in the proofs of the earlier lemmas show that we only have the following possibilities for a holomorphic curve $u$ whose only positive puncture is asymptotic to $\gamma_1$:
\begin{enumerate}
\item The finite energy plane constructed in the proof of Lemma~\ref{lemma:existence_finite_energy_plane}.
\item A holomorphic cylinder $u$ representing an element in some moduli space ${\mathcal M}_0(\gamma_1;\gamma)$, where $\lk(P\times \{ 0 \},\gamma)=1$.
\end{enumerate}

Suppose that $u$ is a holomorphic cylinder as in the second case.
Since $k>1$ and $i$ induces an injection on $\pi_1$, Lemma~\ref{lemma:index_control_2} applies.
As we have already seen in the proof of the previous two lemmas, the orbit $\gamma$ at the negative puncture of $u$ cannot lie in the orbit space $S_{1,0}$.
This leaves the second case from Lemma~\ref{lemma:index_control_2}, which tells us that the free homotopy classes $[\gamma_1]$ and $[\gamma]$ in $Y-P\times \{ 0\}$ are not equal.
Since the cylinder $u$ cannot intersect the binding, the projection of $u$ to $Y-P\times \{ 0\}$ provides a homotopy from $\gamma_1$ to $\gamma$, which is a contradiction.

We conclude that $u$ cannot be a non-trivial holomorphic cylinder, so any rigid $u$ is a translation of $u_0$.
\end{proof}

Finally, we consider the case analogous to that of Lemma~\ref{lemma:fibered_twist_neg_c_single_plane}, but with $c>0$.
Here we use an index argument, which depends on more delicate details, and only works for the following specific, but important examples.
\begin{lemma}
\label{lemma:TKPn}
Let $W=T^*\H \P^m$ or $W=T^*Ca\P^2$.
These manifolds admit right-handed fibered Dehn twists, which we denote by $\tau$.
Define $Y=\OB(W,\tau^{-1})$.
Then the only rigid holomorphic curve in the symplectization of $Y$ with positive puncture $\gamma_1$ is the finite energy plane constructed in the proof of Lemma~\ref{lemma:existence_finite_energy_plane}.
\end{lemma}

\begin{proof}
Let $u$ be a holomorphic curve of genus $0$ with $\gamma_1$ as its only positive puncture. 
The same linking argument we used earlier shows that $u$ can have at most one negative puncture, at which it is asymptotic to $\gamma$.
As before, we can exclude the case that $\gamma$ lies in a fattened neighborhood of the boundary.

As in Lemma~\ref{lemma:fibered_twist_neg_c_single_plane} we will use an index argument.
Denote the contact structure on $Y$ by $\xi$.
Then we can verify that $c_1(\xi)=0$ by using the same method as in the proof of Lemma~\ref{lemma:fibered_twist_neg_c_single_plane}; the new ingredient is $c_1(T^*M,d\lambda_{can})=0$.
By Dragnev's theorem and Lemma~\ref{lemma:index_control_left_twist}, we see that the moduli space $\mathcal M_0(\gamma_1;\gamma)$ is a smooth orbifold of dimension determined by the Fredholm index, and with Formula~\eqref{eq:dimension_formula} we find
$$
\dim \mathcal M_0(\gamma_1;\gamma)=\bar \mu(\gamma_1)-\bar \mu(\gamma)= 2c+1-\dim W+\ind_a f_{convex}.
$$
For $T^*\H \P^m$, we have $c=2m+1$ by \cite[Proposition 2.10]{Frauenfelder:Volume_growth_Dehn}.
Furthermore, $T^*\H \P^m$ admits a plurisubharmonic Morse function with indices $0,4,\ldots,4m$.
It follows that 
$$
\dim \mathcal M_0(\gamma_1;\gamma)=1+4m+2-8m+4r=3-4m+4r,
$$
with $r=0,\ldots,m$. 
By regularity, it follows that the moduli spaces $\mathcal M_0(\gamma_1;\gamma)$ are empty or satisfy $\dim \mathcal M_0(\gamma_1;\gamma)>1$.
It follows that holomorphic cylinders with $\gamma_1$ at its positive puncture cannot be rigid.

For $T^*Ca\P^2$, this argument works as well. Use that $c=11$, and that there is a plurisubharmonic Morse function with indices $0,8,16$.
Then
$$
\dim \mathcal M_0(\gamma_1;\gamma)=2c+1-32+8r=8r -9,
$$
for $r=0,\ldots,2$, so the same argument works.
\end{proof}

\section{Proof of the main theorem and discussion}
\label{sec:wrapup}
All that is left is combining the pieces we set up earlier.
\begin{proof}[Proof of Theorem~\ref{thm:result}]
We apply Lemma~\ref{lemma:non-fillability} to the contact manifold $Y=\OB(W,\tau^{-1})$.
Let us briefly check the conditions.
By Lemma~\ref{lemma:existence_rigid_regular_plane} there is an adjusted almost complex structure on $\R \times Y$ and a rigid finite energy holomorphic plane $u_0$ that is asymptotic to a non-degenerate, simple periodic Reeb orbit $\gamma_1$.
In addition, this plane is unique up to translation and it is Fredholm regular, so the first condition holds.

The second condition of Lemma~\ref{lemma:non-fillability} follows 
\begin{itemize}
\item from Lemma~\ref{lemma:no_other_curves_fractional_fibered} for a fractional twist.
\item from Lemma~\ref{lemma:TKPn} for a fibered twist on $W=T^* \H \P^m$ or $W=T^*Ca \P^2$ and the following observation.
In both cases, the necessary plurisubharmonic Morse function on $W$ comes from a Morse function $f_0$ on the zero-section, and by adding a constant, we can ensure $\max_x(f_0(x) \, )/\min_x(f_0(x) \, )<2$.
This implies the action condition in Lemma~\ref{lemma:non-fillability}.

\item from Lemma~\ref{lemma:only_plane_with_pi1_cond} if $k>1$ and the inclusion $P\to W$ induces an injection on $\pi_1$.
\item from Lemma~\ref{lemma:fibered_twist_neg_c_single_plane} in the remaining case.
\end{itemize}

We conclude that the conditions of Lemma~\ref{lemma:non-fillability} hold for the stable Hamiltonian structure $(\lambda,\Omega_{sH}=d\lambda)$.
This proves the theorem.
\end{proof}

Corollary~\ref{cor:weinstein_conj}, which asserts that the Weinstein conjecture holds for these manifolds, follows from the last statement in Lemma~\ref{lemma:non-fillability}.

We also want to mention another corollary, namely that the negative stabilization is a special case of our construction.
\begin{corollary}[\cite{BvK,Massot:weak_strong_fillability}]
The negative stabilization of the standard contact sphere $(S^{2n+1},\xi_0)$ admits no weak semi-positive symplectic filling.
\end{corollary}
We already described the $S^1$-invariant contact structure on the negative stabilization in Proposition~\ref{prop:negative_stabilization_as_invariant_contact}.
We directly obtain the non-existence of a convex semi-positive symplectic filling from Theorem~\ref{thm:result}.
More work is required to exclude weak fillings: the stable Hamiltonian structure has to be chosen with care. We refer to \cite[Proof of Theorem~3.1]{Massot:weak_strong_fillability}.

\subsection{Negative powers of fractional twists}
To deal with powers of fractional twists we use cobordism techniques due to Avdek.
The following proposition is a special case of \cite[Theorem 1.9]{Avdek:Liouville}.
\begin{proposition}[Avdek]
\label{prop:cobordism_open_book}
Let $W$ be a Weinstein domain, and suppose that $\psi_1$ and $\psi_2$ are symplectomorphisms of $W$ with compact support.
Then there is a Stein cobordism from the disjoint union $\OB(W,\psi_1)\coprod \OB(W,\psi_2)$ at the concave end to $\OB(W,\psi_1 \circ \psi_2)$ at the convex end.
\end{proposition}

The following lemma is essentially contained in \cite{NP:resolving_orbi_singularities}. 
Furthermore, that paper also deals with prequantization bundles over symplectic orbifolds.
For us the following suffices.
\begin{lemma}
\label{lemma:convex_concave_filling}
Let $P$ be a smooth prequantization bundle over a symplectic manifold $(Q,k\omega)$ whose associated disk-bundle is a convex or a concave filling. Then $P$ admits both a convex and a concave filling.
\end{lemma}

\begin{proof}
Perform a symplectic cut on $P\times_{S^1}D^2$.
We obtain a new symplectic manifold $E=Q\tilde \times S^2$, which is an $S^2$-bundle over $Q$. This bundle $E$ contains the smaller copy of the original disk bundle $P\times_{S^1}D^2$ with its given symplectic form as a subset.
This smaller copy still has concave or convex boundary, and the complement of this subset forms then a convex (if the original bundle $P\times_{S^1}D^2$ was concave) or concave (if $P\times_{S^1}D^2$ was convex) filling. 
\end{proof}
We will now give a criterion to test whether the convex fillings obtained this way are semi-positive.
For this, we first set up some notation.
Let $W^{2n-2}$ be a Liouville domain with boundary $P$, where $\pi:P\to (Q,k\omega)$ is a prequantization bundle over a symplectic manifold, with $\omega$ a primitive symplectic form and $k\in \Z_{>1}$.
Suppose that $\tilde P\to P$ is an $\ell$-fold cover of the same form as in Equation~\eqref{eq:cover_BW_bdl}, and assume that $\tilde W$ is an adapted $\ell$-fold cover of $W$.
Denote the right-handed fractional twist of power $\ell$ by $\tilde \tau$.
Consider the contact open book $\tilde Y_+=\OB(\tilde W,\tilde \tau)$.
By Lemma~\ref{lemma:decomposition} $\tilde Y_+$ is contactomorphic to a prequantization bundle over $M=P\times_{S^1,+} D^2 \cup_\partial W$.
Here $(p,z)\in P\times D^2 \sim (pg,gz)\in P\times D^2$ is the equivalence relation for the associated bundle $P\times_{S^1,+} D^2$.

We rescale the symplectic form on $M$ by positive number in order to obtain a primitive symplectic form $\omega_M$.

\begin{lemma}
\label{lemma:fillability_BW}
The above contact manifold $\tilde Y_{+}$ is convex symplectically fillable by the associated disk bundle $\bar L:=\tilde Y_{+}\times_{S^1,-}D^2$, where $(p,z)\sim_{S^1,-} (pg,g^{-1}z)$.
\begin{itemize}
\item if $n=2,3$, then $\bar L$ is trivially semi-positive.
\item if $n\geq 4$, $W$ is Weinstein manifold, and $c_1(Q)=c[\omega]$, then
\[
\begin{split}
c_1(\bar L)&=-\frac{k}{\ell} \pi_{\bar L}^* [\omega_M]\\
c_1(T\bar L)&=(c+k-\frac{k}{\ell})\pi_{\bar L}^*[\omega_M],
\end{split}
\]
where $\pi_{\bar L}:\bar L\to M$ is the natural projection.
In particular $\bar L$, as a symplectic manifold, is semi-positive if $c+k-\frac{k}{\ell}\geq 0$ or if $c+k-\frac{k}{\ell}<3-n$.
\end{itemize}
\end{lemma}

\begin{proof}
We already know that $\tilde Y_{+}$ admits a concave filling by a disk bundle $L=\tilde Y_{+}\times_{S^1,+}D^2$.
Here the equivalence relation for defining $L$ is $(p,z)\sim_+(pg,gz)$, and the symplectic form is $\omega_{Biran}=-d\left( (1-r^2)\theta+r^2d\phi\right)$ as defined by Biran, \cite{Biran:Lagrangian_barrier}.
We give the line bundle associated with $L$ a complex structure, which we define by $j\cdot [p,z]_{+}=[p,iz]_+$.

By Lemma~\ref{lemma:convex_concave_filling} we obtain a convex filling $\bar L$, which can be identified with the disk-bundle we defined in the claim of the lemma.
Furthermore, $\bar L$ inherits a complex structure from the symplectic cutting construction, which we denote by $\bar j$.
Note that $(\bar L,\bar j)$ is dual to $(L,j)$.

We now check the statement concerning the case $n\geq 4$.
We need to compute $c_1(T\bar L)$. Since $\bar L$ can be equipped with a split complex structure, $c_1(T\bar L)=\pi_L^*c_1(\bar L)+\pi_L^*c_1(TM)$, where $\pi_L:\bar L \to M$ is the natural projection.
As $(\bar L,\bar j)$ is dual to $(L,j)$, we see that $c_1(\bar L)=-c_1(L)$.
By Lemma~\ref{lemma:decomposition} we have the decomposition see $M=P\times_{S^1}D^2 \cup_{\partial} W$.
The sequence of the pair gives
$$
H^1(P\times_{S^1,+}D^2) \longrightarrow H^2(M,P\times_{S^1,+}D^2)
\longrightarrow H^2(M) \stackrel{i^*}{\longrightarrow} H^2(P\times_{S^1,+}D^2 ) \longrightarrow H^3(M,P\times_{S^1,+}D^2),
$$
so by excision, the assumption that $W$ is Weinstein and the dimension restriction we see that $H^2(M,P\times_{S^1,+}D^2)\cong H^2(W,\partial W)\cong H^{2n-4}(W)=0$, and $i^*$ is an injection.
The same ingredients imply also that $H^3(M,P\times_{S^1,+}D^2)$ has no torsion.
This means that $[\pi^* \omega]$ lies in the image of $i^*$, so we conclude that $[\omega_M]=(i^*)^{-1}[\pi^* \omega]$ as these elements are primitive.
It is therefore enough to compute $i^*c_1(L)$ and $i^*c_1(TM)$.

By the assumption that $P$ is a prequantization bundle over $(Q,k\omega)$, we see that as a line bundle over $Q$, we have $c_1(P\times_{S^1,+}D^2)=k[\omega]$.
Hence we compute
\[
i^*c_1(TM)=c_1(T(P\times_{S^1,+}D^2)\,)=\pi^*c_1(P\times_{S^1,+}D^2)+\pi^*c_1(TQ)=(c+k)\pi^*[\omega].
\]
Since $\tilde P$ is the $\ell$-fold cover of $P$, we have
\[
i^*c_1(L)=\frac{1}{\ell}c_1(P\times_{S^1,+}D^2)=\frac{k}{\ell}\pi^*[\omega].
\]
The claim about the first Chern class follows, and the statement about semi-positivity follows directly from the definition.
\end{proof}

\begin{corollary}
\label{cor:no_convex_filling}
Let $W^{2n-2}$ be a Weinstein domain admitting a right-handed fractional twist $\tau$ such that $\OB(W,\tau^{-1})$ is not convex semi-positively fillable.
Suppose that the prequantization bundle $\OB(W,\tau)$ is convex semi-positively fillable.
Then for all positive integers $N$, the contact manifold $\OB(W,\tau^{-(N+1)})$ is not convex semi-positively fillable.
\end{corollary}

\begin{proof}
We argue by contradiction to prove the assertion.
Suppose that $\OB(W,\tau^{-(N+1)})$ is convex symplectically fillable by $F_{-(N+1)}$.
By assumption $\OB(W,\tau)$ is convex semi-positively symplectically fillable, and we call this filling $F_1$.

\begin{figure}[htp]
\def\svgwidth{0.75\textwidth}%
\begingroup\endlinechar=-1
\resizebox{0.75\textwidth}{!}{%
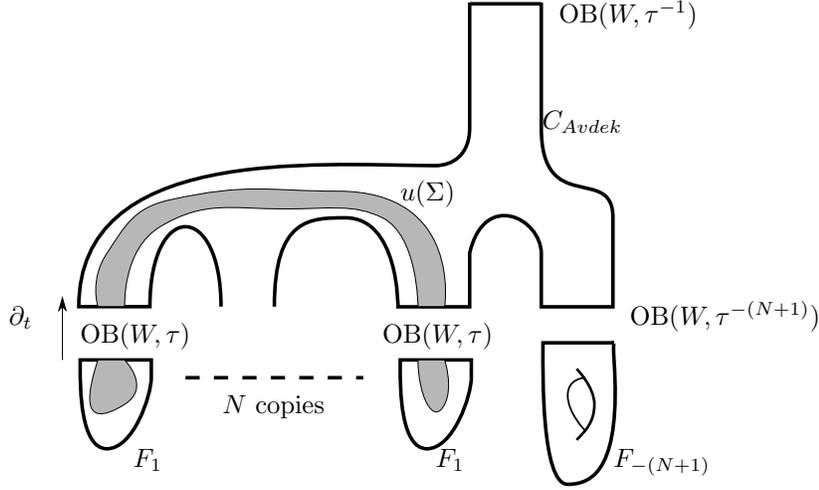%
}\endgroup
\caption{Capping off the cobordism}
\label{fig:cobordism}
\end{figure}

By Proposition~\ref{prop:cobordism_open_book}, there is an exact symplectic cobordism $(C_{Avdek},d\lambda_{Avdek})$ whose convex end is $\OB(W,\tau^{-1})$, and whose concave ends are $N$ copies of $\OB(W,\tau)$ and a single copy of $\OB(W,\tau^{-(N+1)})$.
Cap off the concave ends using the convex fillings we obtained, and we find a convex symplectic filling for $\OB(W,\tau^{-1})$.
We can make the almost complex structure standard in the sense that in the gluing regions involving the $N$ copies of $\OB(W,\tau)$ and $\OB(W,\tau^{-(N+1)})$ the almost complex structure sends $\partial_t$ to the Reeb fields of the corresponding ends.

Semi-positivity does, in general, not hold in this glued cobordism, but we shall argue that we still have control over sphere bubbles.
Indeed, let $u$ be a parametrized holomorphic sphere that is not completely contained in one of the $F_1$'s or in $F_{-(N+1)}$.
Then we find a subset $\Sigma\subset S^2$ such that $u(\Sigma)\subset C_{Avdek}$, and by Stokes' theorem we compute the energy as
$$
\int_{\Sigma}u^*d \lambda_{Avdek}=\int_{\partial \Sigma} u^*\lambda_{Avdek}.
$$
For $J$-holomorphic curves, the energy must be non-negative.
However, the $\partial_t$-component of $(Tu) \nu$ (here $\nu$ is an outward normal to $\Sigma$) is negative, so by our choice of almost complex structure we find $\int_{\partial \Sigma} u^*\lambda_{Avdek}<0$.
We conclude that any sphere bubble must be contained in one of the $F_1$'s or in $F_{-(N+1)}$.
Since these symplectic manifolds are semi-positive by assumption, we conclude that sphere bubbles cannot occur.

To conclude, in the filling for $\OB(W,\tau^{-1})$ we constructed above, sphere bubbles cannot occur, so the argument in the proof of Lemma~\ref{lemma:non-fillability} will give a contradiction to the existence of a convex, semi-positive filling of $\OB(W,\tau^{-(N+1)})$.
\end{proof}

Note that $\OB(W,\tau^N)$ is a prequantization bundle over a symplectic orbifold by arguments in \cite{CDvK:right-handed}.
By \cite{NP:resolving_orbi_singularities}, this manifold is convex symplectically fillable, so instead of attaching $N$ convex fillings for $\OB(W,\tau)$, we could also have used a single convex filling for $\OB(W,\tau^N)$.
However, if one does so, sphere bubbles are more difficult to control.
In this case, we can also find a convex filling with the cobordism techniques of Avdek.
\begin{lemma}
Let $W$ be a Weinstein domain admitting a right-handed fractional twist $\tau$, and suppose that $\OB(W,\tau)$ admits a convex, semi-positive symplectic filling.
Then for all $N\in \Z_{>0}$, the contact manifold $\OB(W,\tau^N)$ admits a convex, semi-positive symplectic filling.
\end{lemma}
\begin{proof}
By Avdek's Proposition~\ref{prop:cobordism_open_book}, there is an exact symplectic cobordism with $\OB(W,\tau^N)$ at the convex end and $N$ copies of $\OB(W,\tau)$ at the concave ends.
Cap off each concave end with a convex semi-positive filling for $\OB(W,\tau)$.
\end{proof}

We collect the above results as Theorem~\ref{thm:symplectic_isotopy} mentioned in the introduction.
\begin{theorem}
\label{thm:powers_of_twists}
Let $W$ be a Weinstein domain admitting a right-handed fractional twist $\tau$ of power $\ell$, and suppose that $\OB(W,\tau^{-1})$ has no convex, semi-positive filling.
Assume that $Y_+=\OB(W,\tau)$ admits a convex, semi-positive symplectic filling.
Then for all $N \in \Z_{>0}$, the contact manifold $\OB(W,\tau^{-N})$ is not convex, semi-positively fillable.

In particular, $\tau^N$ is not symplectically isotopic to the identity relative to the boundary.
\end{theorem}
The first claim follows directly from Corollary~\ref{cor:no_convex_filling}.
For the last claim, observe that if $\tau^{-N}$ is symplectically isotopic to the identity relative to the boundary, then $\OB(W,\tau^{-N})$ is contactomorphic to $\OB(W,\id)$ which is fillable by the Weinstein domain $W\times D^2$.
This gives a contradiction, so the assertion holds.

\subsection{Algebraic overtwistedness}
\label{sec:HC=0}
The discussion in this section will be of a conjectural nature, since we will assume that contact homology algebra exists, and has the expected properties. For a more detailed discussion of this point of view, see \cite{BvK}.

With this in mind, we claim that our construction also shows that contact homology algebra will vanish if the fillability obstructions listed in Theorem~\ref{thm:result} hold.

Let $A_*(Y,\alpha;\Q[H_2(Y)])$ denote the contact homology algebra chain complex with full coefficients in $H_2(Y;\Z)$.
We claim that $\gamma_1\in A_*(Y,\alpha;\Q[H_2(Y)])$ satisfies $\partial \gamma_1 =\pm 1$.

Indeed, the linking argument from the proof of Lemma~\ref{lemma:no_other_curves_fractional_fibered} shows that the holomorphic curve count needed for $\partial \gamma_1$ only involves cylinders and planes.
By Lemma~\ref{lemma:existence_rigid_regular_plane}, there is a unique plane, so $\partial \gamma_1 =\pm 1 +\sum_i n_i \gamma^-_i$.
Lemmas \ref
{lemma:no_other_curves_fractional_fibered}
\ref{lemma:fibered_twist_neg_c_single_plane}, \ref{lemma:only_plane_with_pi1_cond} and \ref{lemma:TKPn} show that there are no rigid holomorphic cylinders.
Note here that the contact manifolds in the Lemmas that rely on index arguments, namely Lemma~\ref{lemma:fibered_twist_neg_c_single_plane} and \ref{lemma:TKPn}, always have $c_1(\xi)=0$.

\subsubsection{Left-handed stabilizations}
We also recover the main result of \cite{BvK}, which asserts that certain left-handed stabilizations have vanishing contact homology, by the following argument.

Let $Y_{-}=\OB(T^*S^n,\tau^{-1}_{Dehn})$, and let $Y'$ be another closed contact manifold of the same dimension.
By the condition~\eqref{eq:condition_C} we imposed on the function $U_0$ and Lemma~\ref{lemma:orbits_near_binding} the orbit $\gamma_1$ in $Y_-$ has minimal action.
By rescaling the contact form in $Y'$, we can assume that all periodic Reeb orbits in $Y'$ have much larger action than that of $\gamma_1$.
Now take the connected sum $Y_- \# Y'$ along suitable Darboux balls using the Weinstein model.
New orbits are created, but by shrinking the connecting tube of the Weinstein model most of the new periodic Reeb orbits have large action, except possibly for the periodic orbits contained in the connecting tube of the connected sum; these orbits are also known as Lyapunov orbits.

Since these Lyapunov orbits have reduced index $2n-3,2n-1,\ldots$ and $\gamma_1$ has reduced index equal to $1$, it follows that there is no holomorphic curve from $\gamma_1$ to any combination of these Lyapunov orbits. Since all other orbits have larger action than $\gamma_1$, it follows that $\partial \gamma_1=\pm 1$ even after the connected sum.

However, contact homology is still a work in progress, mainly due to transversality problems involving holomorphic curves that are not simple, and we will not pursue this argument.

\subsection{Generalizations and questions}
There are some obvious generalizations which we haven't worked out.
For example, consider a Weinstein domain for which the Reeb flow on the boundary is periodic, say $\partial W=P$.
If we mod out $P$ by its Reeb action, we obtain a symplectic orbifold.
We can still define fractional twists, and many arguments will still go through.
See for instance \cite[Section 4.1.1]{vanKoert:openbooks5} for a variation of the fractional twist in the case that the boundary of $W$ is a Brieskorn manifold, which is a prequantization bundle over a symplectic orbifold.

\subsubsection{$T^*\C P^n$ as a page}
The index argument we relied on for $T^*\H \P^m$ and $T^*Ca\P^2$ does not work in this case.
However, on the odd complex projective spaces there is a free involution, namely
\[
\begin{split}
\sigma: \C \P^{2n+1} & \longrightarrow \C \P^{2n+1} \\
[z_0:z_1:\ldots:z_{2n}:z_{2n+1}] & \longmapsto [-\bar z_1:\bar z_0:\ldots:-\bar z_{2n+1}:\bar z_{2n}].
\end{split}
\]
The Fubini-Study metric on complex projective space has the property that all geodesics are periodic with the same period.
It follows that $M:=\C \P^{2n+1}/\sigma$ admits a metric for which all geodesics are periodic.
Unfortunately, the periods are not all the same unless we are looking at $T^*\C \P^1/\sigma\cong T^*\R \P^2$, so this involution does not help us in general.

In the latter case we can find a fractional twist $\tilde \tau$ on $T^*\C P^{1}$ of power $2$, which is the usual Dehn twist on $T^*S^2$.
By the first part of Theorem~\ref{thm:result}, it follows that $\OB(T^*\C P^{1},\tilde \tau^{-1})$ is not convex semi-positively fillable.
On the other hand, the contact open book $Y_{+}=\OB(T^*\C P^{1},\tilde \tau)\cong S^5$ is Liouville fillable.
Theorem~\ref{thm:powers_of_twists} then tells us that $\OB(T^*\C \P^{1},\tilde \tau^{-N})$ is not convex semi-positively fillable.

\begin{question}[AIM Workshop]
Let $\tau$ denote a right-handed fibered Dehn twist on $T^*\C \P^n$.
Is $\OB(T^*\C P^n,\tau^{-1})$ non-fillable for $n\geq 2$?
\end{question}

\end{document}